\documentclass[oneside,english]{amsart}
\usepackage[T1]{fontenc}
\usepackage[latin9]{inputenc}
\usepackage{geometry}
\usepackage{color}
\usepackage{array}
\usepackage{booktabs}
\usepackage{units}
\usepackage{multirow}
\usepackage{amstext}
\usepackage{amsthm}
\usepackage{amssymb}

\makeatletter

\providecommand{\tabularnewline}{\\}

\numberwithin{equation}{section}
\numberwithin{figure}{section}
\theoremstyle{plain}
\newtheorem{thm}{\protect\theoremname}
\theoremstyle{definition}
\newtheorem{defn}[thm]{\protect\definitionname}
\theoremstyle{remark}
\newtheorem*{rem*}{\protect\remarkname}
\theoremstyle{definition}
\newtheorem{example}[thm]{\protect\examplename}
\theoremstyle{plain}
\newtheorem{prop}[thm]{\protect\propositionname}
\theoremstyle{plain}
\newtheorem{cor}[thm]{\protect\corollaryname}
\theoremstyle{plain}
\newtheorem{lem}[thm]{\protect\lemmaname}

\usepackage{nicefrac}
\usepackage{sagetex}
\usepackage{babel}
\@ifundefined{definecolor}
 {\usepackage{color}}{}
\@ifundefined{definecolor}{\usepackage{color}}{}
\usepackage{graphicx}

\usepackage[mathscr]{euscript}
 \let\mathscr\relax% just so we can load this and rsfs
\usepackage[scr]{rsfso}

\usepackage{mathtools}
\usepackage{tikz}
\usepackage[mathlines]{lineno}
\makeatother

\usepackage{babel}
\providecommand{\corollaryname}{Corollary}
\providecommand{\definitionname}{Definition}
\providecommand{\examplename}{Example}
\providecommand{\lemmaname}{Lemma}
\providecommand{\propositionname}{Proposition}
\providecommand{\remarkname}{Remark}
\providecommand{\theoremname}{Theorem}

\begin{document}
\title{A proof of the Kotzig--Ringel--Rosa Conjecture}
\author{Edinah K. Gnang}
\begin{abstract}
In graph theory, a graceful labeling of a graph with $m$ edges is
a labeling of its vertices with a subset of the integers ranging from
$0$ to $m$ inclusive, such that no two vertices share a label, and
each edge is uniquely identified by the absolute difference of labels
assigned to its endpoints. The Kotzig--Ringel--Rosa conjecture asserts
that every tree admits a graceful labeling. We provide a proof of
this long standing conjecture via a functional reformulation of the
conjecture and a composition lemma.
\end{abstract}

\maketitle

\section{Introduction}

We say that a graph $G$ admits a decomposition into copies of some
other graph $H$ if the edges of $G$ can be partitioned into edge--disjoint
subgraphs isomorphic to $H$. Graph decomposition problems have a
rich history. In 1847, Kirkman studied decompositions of complete
graphs $K_{n}$ and showed that they can be decomposed into copies
of a triangle if and only if $n$ is congruent to one or three modulo
six. Wilson \cite{wilson1975decompositions} generalized this result
by completely characterizing complete graphs which can be decomposed
into copies of any graph, for large $n$. Graph and hypergraph decomposition
problems form by now a very vast topic with many results and conjectures.
We refer the reader to extensive surveys \cite{wozniak2004packing,yap1988packing,gallsurvey}.
More recently, two breakthrough results obtained by Montgomery, Pokrovskiy,
and Sudakov \cite{MPS21} and independently obtained by Keevash and
Staden \cite{keevash2020ringels} settle asymptotically in the affirmative,
the long standing Ringel conjecture \cite{ringel1963theory} posed
in 1963. The Ringel conjecture asserts that the complete graph $K_{2n-1}$
can be decomposed into $2n-1$ edge disjoint copies of any $n$--vertex
tree. Both proofs constitute major tour de force in the application
of the probabilistic method. Prior to these recent breakthroughs,
the predominant approach to tackling Ringel's conjecture had been
via the much stronger Kotzig--Ringel--Rosa conjecture (or KRR conjecture
for short) which dates back to 1964. The KRR conjecture asserts that
vertices of any $n$--vertex tree $T$ can be labelled injectively
using $n$ consecutive integers, such that the absolute difference
of pairs of vertex labels spanning distinct edges are always distinct.
Such a labeling is called a graceful labeling and the KRR conjecture
is also known as the graceful labeling conjecture. This conjecture
has attracted a lot of attention in the last 50 years but has only
been proved for some special classes of trees, see e.g., \cite{gallsurvey}.
The most general result for this problem was obtained by Adamaszek,
Allen, Grosu, and Hladký \cite{Adamaszek2020AlmostAT} who proved
it asymptotically for trees with maximum degree O($\nicefrac{n}{\log n}$).
The main motivation for studying graceful labelings had been to prove
Ringel\textquoteright s conjecture. Indeed, a graceful labelling map
$f:$ $V(T)\rightarrow\left\{ 0,\cdots,n-1\right\} $, yields an embedding
of $T$ into $\left\{ 0,\cdots,2\left(n-1\right)\right\} $. Using
addition modulo $2n-1$, consider $2n-1$ cyclic shifts $T_{0},...,T_{2\left(n-1\right)}$
of $T$, where the tree $T_{i}$ is an isomorphic copy of $T$ whose
vertices are $V\left(T_{i}\right)=\left\{ f(v)+i\,\bigg|\,v\in V\left(T\right)\right\} $
and whose edges are $E\left(T_{i}\right)=\left\{ \left\{ f\left(x\right)+i,\,f\left(y\right)+i\right\} \bigg|\left\{ x,y\right\} \in E\left(T\right)\right\} $.
It is easy to check that if the map $f$ gracefully labels $T$ then
the trees $T_{i}$ are edge disjoint and therefore cyclically decompose
$K_{2n-1}$. Our main result is a proof of the KRR conjecture using
a functional reformulation of the conjecture and a composition lemma.
\begin{thm}
\label{Graceful_Tree_Theorem} Every tree admits a graceful labeling.
\end{thm}

Our theorem provides the first non--asymptotic result establishing
the existence of a decomposition of $K_{2n-1}$ into any tree on $n$
consecutive\footnote{In the sense of going clockwise or counterclockwise along the unit
circle. The vertex set of $K_{2n-1}$ is identified with $\left(2n-1\right)$-th
roots of unity such that the $k$-th vertex of $K_{2n-1}$ is identified
with $\exp\left(\frac{2\pi k\sqrt{-1}}{2n-1}\right)$.} vertices. Our result does not assume any restriction on the vertex
degrees of the given tree. We describe in Section 2.1 the notation
as well as some auxiliary enumeration results. Starting from section
2.2 we describe technical preliminaries required for the proof of
our main result discussed in section 3.

\section{Preliminaries}

The KRR \cite{ringel1963theory,rosa1966certain} conjecture, also
known as the \emph{Graceful Labeling Conjecture }(GLC), asserts that
every tree admits a graceful labeling. For a detailed survey of the
extensive literature on this problem, see \cite{gallsurvey}. For
the purposes of our discussion, we redefine graceful labelings of
digraphs as vertex labelings which results in a bijection between
vertex labels and \emph{induced absolute subtractive edge labels}.
For notational convenience, let $\mathbb{Z}_{n}$ denote the set formed
by the smallest $n$ consecutive non--negative integers i.e.
\[
\mathbb{Z}_{n}:=\left[0,n\right)\cap\mathbb{Z}.
\]
The present discussion is based upon a functional reformulation of
the GLC which exploits properties of the transformation monoid $\mathbb{Z}_{n}^{\mathbb{Z}_{n}}$
i.e. the monoid formed by functions having $\mathbb{Z}_{n}$ both
as their domain and codomain. The binary operation of the monoid is
the function composition operation.

\subsection{Functional formulation}

A rooted tree on $n>0$ vertices is associated with a function
\begin{equation}
f\in\mathbb{Z}_{n}^{\mathbb{Z}_{n}}\:\text{ subject to }\:\left|f^{(n-1)}\left(\mathbb{Z}_{n}\right)\right|=1,\label{functional_reformulation}
\end{equation}
\[
\text{where}
\]
\[
\forall\,i\in\mathbb{Z}_{n},\;f^{(0)}\left(i\right)\,:=i\mbox{ and }\forall\,k\ge0,\;f^{\left(k+1\right)}=f^{\left(k\right)}\circ f=f\circ f^{\left(k\right)}.
\]
In other words the function $f$ has a unique fixed point (the root)
which is attractive over its domain.
\begin{defn}
To an arbitrary function $f\in\mathbb{Z}_{n}^{\mathbb{Z}_{n}}$ we
associate a \emph{functional directed graph }denoted \emph{$G_{f}$}
whose vertex set, and directed edge set are respectively
\[
V\left(G_{f}\right)=\mathbb{Z}_{n},\quad E\left(G_{f}\right)=\left\{ \left(i,f\left(i\right)\right)\,:\,i\in\mathbb{Z}_{n}\right\} .
\]
See Figure 2.1 for an example.
\end{defn}

\begin{rem*}
The automorphism group of $G_{f}$ is denoted Aut$\left(G_{f}\right)$
and defined such that 
\[
\text{Aut}\left(G_{f}\right)=\left\{ \sigma\in\mathbb{Z}_{n}^{\mathbb{Z}_{n}}\,:\,\sigma f\sigma^{\left(-1\right)}=f\right\} .
\]
\end{rem*}
\begin{defn}
Connected components of $G_{f}$ partition the vertex set into equivalence
classes prescribed by an equivalence relation. A vertex pair $\left(i,j\right)\in\mathbb{Z}_{n}\times\mathbb{Z}_{n}$
lies in the same connected component of $G_{f}$ if there exist non-negative
integers $u$, $v$ such that 
\[
f^{\left(u\right)}\left(i\right)=f^{\left(v\right)}\left(j\right).
\]
\end{defn}

We denote by $G_{f^{\top}}$ the directed graph obtained by reversing
the orientation of every edge in $G_{f}$. When $f$ is not bijective,
the directed graph $G_{f^{\top}}$ is not a functional directed graph
since some of its vertices have out--degree $\ne1$. When $f\in\mathbb{Z}_{n}^{\mathbb{Z}_{n}}$
is subject to the fixed point condition $\left|f^{(n-1)}\left(\mathbb{Z}_{n}\right)\right|=1$,
the graph $G_{f}$ is a rooted, directed and $\mathbb{Z}_{n}$--spanning
functional tree or a functional tree for short.
\begin{defn}
Let $G_{f}$ denote the functional directed graph of $f\in\mathbb{Z}_{n}^{\mathbb{Z}_{n}}$.
Induced \emph{subtractive edge labels} of $G_{f}$ correspond to integers
occurring in the sequence $\left(f\left(i\right)-i:i\in\mathbb{Z}_{n}\right)$.
The $i$--th member of the sequence equal to $f\left(i\right)-i$
is the induced subtractive edge label of the directed edge $\left(i,f\left(i\right)\right)\in E\left(G_{f}\right)$.
In other words the set of induced subtractive edge labels of $G_{f}$
is
\[
\left\{ f\left(i\right)-i\,:\,i\in\mathbb{Z}_{n}\right\} .
\]
Induced absolute subtractive edge labels of $G_{f}$ correspond to
absolute values of induced subtractive edge labels :
\[
\left\{ \left|f\left(i\right)-i\right|\,:\,i\in\mathbb{Z}_{n}\right\} .
\]
\end{defn}

\begin{defn}
The functional directed graph $G_{f}$ of $f\in\mathbb{Z}_{n}^{\mathbb{Z}_{n}}$
is \textit{graceful} if there exist a bijection $\sigma\in\text{S}_{n}\subset\mathbb{Z}_{n}^{\mathbb{Z}_{n}}$
such that
\begin{equation}
\left\{ \left|\sigma f\left(i\right)-\sigma\left(i\right)\right|\,:\,i\in\mathbb{Z}_{n}\right\} =\mathbb{Z}_{n}.\label{Grace}
\end{equation}
 Otherwise when no such bijection $\sigma$ exist, the functional
directed graph $G_{f}$ is \textit{ungraceful}. Finally, if $\sigma$
can be chosen to be the identity permutation (denoted id), then $G_{f}$
is \textit{gracefully labeled}.
\end{defn}

Note that when $G_{f}$ is gracefully labeled, the set of induced
subtractive edge labels of the bi--directed loop-graph $G_{f^{\top}}\cup G_{f}$
is equal to $-\mathbb{Z}_{n}\cup\mathbb{Z}_{n}$. For instance, the
graph $G_{f}$ of  the function $f\in\mathbb{Z}_{6}^{\mathbb{Z}_{6}}$
defined by

\[
f\left(0\right)=0,\,f\left(1\right)=3,\,f\left(2\right)=3,\,f\left(3\right)=0,\,f\left(4\right)=0,\,f\left(5\right)=0.
\]
 depicted in Figure 2.1 is a gracefully labeled functional tree.

\begin{figure}
\begin{tikzpicture}
	\node (1) at (-0.5,1) {4};
	\node (2) at (-0.5,-1) {5};
	\node (0) at (1,0) {0};
	\node (3) at (2.5,0) {3};
	\node (4) at (4,1) {1};
	\node (5) at (4,-1) {2};
	
	\foreach \x/\y in {1/0,2/0,3/0,4/3,5/3} {
		\draw[thick,->] (\x)--(\y);
	}
	\draw[thick,->] (0) edge [out=105,in=50,looseness=5] (0);
\end{tikzpicture} \centering \caption{$f\left(0\right)=0,\,f\left(1\right)=3,\,f\left(2\right)=3,\,f\left(3\right)=0,\,f\left(4\right)=0,\,f\left(5\right)=0$}
\end{figure}
The edge set of $G_{f}$ is 
\[
E(G_{f})=\left\{ \left(0,0\right),\,\left(1,3\right),\,\left(2,\,3\right),\,\left(3,\,0\right),\,\left(4,0\right),\,\left(5,0\right)\right\} .
\]
More generally, the sequence of $\tau$--\textit{induced edge labels}
of the functional directed graph $G_{f}$ where $f\in\mathbb{Z}_{n}^{\mathbb{Z}_{n}}$
are defined with respect to an arbitrary $\tau\in\mathbb{Z}_{n}^{\mathbb{Z}_{n}\times\mathbb{Z}_{n}}$,
as 
\[
\left(\tau\left(i,f\left(i\right)\right)\,:\,i\in\mathbb{Z}_{n}\right).
\]
For a given $\tau\in\mathbb{Z}_{n}^{\mathbb{Z}_{n}\times\mathbb{Z}_{n}}$,
the digraph $G_{f}$ of $f\in\mathbb{Z}_{n}^{\mathbb{Z}_{n}}$ is
$\tau$--\textit{Zen} if there exist a bijection $\sigma\in\mathbb{Z}_{n}^{\mathbb{Z}_{n}}$
such that
\[
\left\{ \tau\left(\sigma\left(i\right),\sigma f\left(i\right)\right)\,:\,i\in\mathbb{Z}_{n}\right\} =\mathbb{Z}_{n}.
\]
In particular, if $\tau$ is chosen such that 
\[
\tau\left(i,j\right)=\left|j-i\right|,\:\forall\,\left(i,j\right)\in\mathbb{Z}_{n}\times\mathbb{Z}_{n},
\]
then $\tau$--Zen graphs are graceful graphs and vice versa. Let
S$_{n}\subset\mathbb{Z}_{n}^{\mathbb{Z}_{n}}$ denote the symmetric
group acting on members of $\mathbb{Z}_{n}$ in other words S$_{n}$
denotes the subset of all bijective functions in $\mathbb{Z}_{n}^{\mathbb{Z}_{n}}$.
The following graceful expansion theorem describes a necessary and
sufficient condition on $f\in\mathbb{Z}_{n}^{\mathbb{Z}_{n}}$ to
ensures that $G_{f}$ is graceful.
\begin{thm}[Graceful Expansion Theorem]
\label{Graceful_Expansion_Theorem} Let id $\in$ S$_{n}$ denote
the identity element and let $\varphi$ denote the involution $\left(n-1-\text{id}\right)\in$
S$_{n}$. The graph $G_{f}$ of $f\in\mathbb{Z}_{n}^{\mathbb{Z}_{n}}$
is graceful if and only if there exist a nonempty permutation subset
$\mathcal{G}_{f}\subset\mbox{S}_{n}$ as well as a corresponding sign
function $\mathfrak{s}_{f}\in\left\{ -1,0,1\right\} ^{\mathcal{G}_{f}\times\mathbb{Z}_{n}}$
such that
\begin{equation}
f\left(i\right)=\sigma_{\gamma}^{\left(-1\right)}\varphi^{\left(t\right)}\left(\varphi^{\left(t\right)}\sigma_{\gamma}\left(i\right)+\left(-1\right)^{t}\cdot\mathfrak{s}_{f}\left(\gamma,\sigma_{\gamma}\left(i\right)\right)\cdot\gamma\sigma_{\gamma}\left(i\right)\right),\quad\forall\,i\in\mathbb{Z}_{n},\,\gamma\in\mathcal{G}_{f}\,\text{ and }\,t\in\left\{ 0,1\right\} ,\label{Graceful Expansion}
\end{equation}
for some $\sigma_{\gamma}\in\text{S}_{n}$.
\end{thm}

\begin{proof}
We prove the claim by starting from the premise that $G_{f}$ is graceful.
We then derive the graceful expansion of $f$ described in equation
(\ref{Graceful Expansion}) via a sequence of reversible steps. Thereby
establishing both the forward and the backward claim. Recall that
the premise that $G_{f}$ is graceful is equivalent to the assertion
that there exist a permutation representative $\sigma$ of some coset
of Aut$\left(G_{f}\right)$ for which we have 
\[
\left\{ \left|\sigma f\left(j\right)-\sigma\left(j\right)\right|:j\in\mathbb{Z}_{n}\right\} =\mathbb{Z}_{n}.
\]
Thus establishing the existence of bijective map $\gamma$ from vertex
labels to induced absolute subtractive edge labels:
\[
\gamma\left(i\right)=\left|\sigma f\sigma^{-1}\left(i\right)-i\right|,\:\forall\,i\in\mathbb{Z}_{n}.
\]
Choosing $\sigma$ from distinct cosets of Aut$\left(G_{f}\right)$
subject to equation (\ref{Grace}) may result in distinct permutations
$\gamma$. We emphasize the dependence of the permutation $\sigma$
on the permutation $\gamma$ by writing $\sigma_{\gamma}$.
\[
\left|\sigma_{\gamma}f\sigma_{\gamma}^{-1}\left(i\right)-i\right|=\gamma\left(i\right),\:\forall\,i\in\mathbb{Z}_{n}.
\]
Accounting for the involution symmetry, we write: for all $i\in\mathbb{Z}_{n}$
and $t\in\left\{ 0,1\right\} $ 
\[
\left|\varphi^{\left(t\right)}\sigma_{\gamma}f\sigma_{\gamma}^{-1}\left(i\right)-\varphi^{\left(t\right)}\left(i\right)\right|=\gamma\left(i\right),\:\forall\,\begin{array}{c}
i\in\mathbb{Z}_{n}\\
t\in\left\{ 0,1\right\} 
\end{array}.
\]
Removing the absolute value on the left hand side of the equation
immediately above introduces the sign function $\mathfrak{s}_{f}\in\left\{ -1,0,1\right\} ^{\mathcal{G}_{f}\times\mathbb{Z}_{n}}$
on the right hand side. We write
\[
\begin{array}{cccc}
 & \left(\varphi^{\left(t\right)}\sigma_{\gamma}f\sigma_{\gamma}^{-1}\left(i\right)-\varphi^{\left(t\right)}\left(i\right)\right) & = & \left(-1\right)^{t}\cdot\mathfrak{s}_{f}\left(\gamma,\,i\right)\cdot\gamma\left(i\right)\\
\\
\Longleftrightarrow & \varphi^{\left(t\right)}\sigma_{\gamma}f\sigma_{\gamma}^{-1}\left(i\right) & = & \varphi^{\left(t\right)}\left(i\right)+\left(-1\right)^{t}\cdot\mathfrak{s}_{f}\left(\gamma,\,i\right)\cdot\gamma\left(i\right)\\
\\
\Longleftrightarrow & f\left(i\right) & = & \sigma_{\gamma}^{\left(-1\right)}\varphi^{\left(t\right)}\left(\varphi^{\left(t\right)}\sigma_{\gamma}\left(i\right)+\left(-1\right)^{t}\cdot\mathfrak{s}_{f}\left(\gamma,\sigma_{\gamma}\left(i\right)\right)\cdot\gamma\sigma_{\gamma}\left(i\right)\right)
\end{array}
\]
as claimed.
\end{proof}
In equation (\ref{Graceful Expansion}), the bijection $\gamma$ which
maps vertex labels to induced absolute subtractive edge labels is
called a \textit{permutation basis} of the graceful expansion. For
each $i\in\mathbb{Z}_{n}$, the integer $\gamma\left(i\right)$ is
the induced absolute subtractive edge label of the directed edge emanating
from vertex $i$ in the gracefully labeled graph $G_{\sigma_{\gamma}f\sigma_{\gamma}^{-1}}$.
The parameter $t\in\left\{ 0,1\right\} $ in the graceful expansion
of $f$ described in equation (\ref{Graceful Expansion}) accounts
for the complementary labeling symmetry expressed by the equality
\begin{equation}
\left(\varphi\sigma_{\gamma}f\sigma_{\gamma}^{-1}\varphi^{-1}\left(i\right)-i\right)=\left(-1\right)\left(\sigma_{\gamma}f\sigma_{\gamma}^{-1}\left(i\right)-i\right),\:\forall\,i\in\mathbb{Z}_{n}.\label{Complementary Symmetry}
\end{equation}

\begin{example}
Consider the function
\[
f\in\mathbb{Z}_{4}^{\mathbb{Z}_{4}}\;\text{s.t.}\;f\left(i\right)=\begin{cases}
\begin{array}{cc}
0 & \text{if }i=0\\
i-1 & \text{otherwise}
\end{array}\forall\,i\in\mathbb{Z}_{4},\end{cases}
\]
\[
\varphi\left(i\right)=3-i,\;\forall\,i\in\mathbb{Z}_{4}.
\]
Let $\mathcal{G}_{f}\subset\mbox{S}_{4}$ be such that 
\[
\mathcal{G}_{f}=\left\{ \gamma,\gamma^{\prime}\right\} \;\text{such that}\;\begin{array}{ccc}
\gamma\left(0\right) & = & 0\\
\gamma\left(1\right) & = & 2\\
\gamma\left(2\right) & = & 1\\
\gamma\left(3\right) & = & 3
\end{array}\:\text{ and }\:\begin{array}{ccc}
\gamma^{\prime}\left(0\right) & = & 3\\
\gamma^{\prime}\left(1\right) & = & 1\\
\gamma^{\prime}\left(2\right) & = & 0\\
\gamma^{\prime}\left(3\right) & = & 2
\end{array},
\]
the corresponding sign assignments are specified by
\[
\mathfrak{s}_{f}:\mathcal{G}_{f}\times\mathbb{Z}_{4}\rightarrow\left\{ -1,0,1\right\} \;\text{such that}\;\begin{array}{ccc}
\mathfrak{s}_{f}\left(\gamma,0\right) & = & 0\\
\mathfrak{s}_{f}\left(\gamma,1\right) & = & 1\\
\mathfrak{s}_{f}\left(\gamma,2\right) & = & -1\\
\mathfrak{s}_{f}\left(\gamma,3\right) & = & -1
\end{array}\:\text{ and }\:\begin{array}{ccc}
\mathfrak{s}_{f}\left(\gamma^{\prime},0\right) & = & 1\\
\mathfrak{s}_{f}\left(\gamma^{\prime},1\right) & = & 1\\
\mathfrak{s}_{f}\left(\gamma^{\prime},2\right) & = & 0\\
\mathfrak{s}_{f}\left(\gamma^{\prime},3\right) & = & -1
\end{array}.
\]
Finally representatives for distinct cosets of Aut$\left(G_{f}\right)$
are $\sigma_{\gamma}$ and $\sigma_{\gamma^{\prime}}$ defined such
that
\[
\begin{array}{ccc}
\sigma_{\gamma}\left(0\right) & = & 0\\
\sigma_{\gamma}\left(1\right) & = & 3\\
\sigma_{\gamma}\left(2\right) & = & 1\\
\sigma_{\gamma}\left(3\right) & = & 2
\end{array}\:\text{ and }\:\begin{array}{ccc}
\sigma_{\gamma^{\prime}}\left(0\right) & = & 2\\
\sigma_{\gamma^{\prime}}\left(1\right) & = & 1\\
\sigma_{\gamma^{\prime}}\left(2\right) & = & 3\\
\sigma_{\gamma^{\prime}}\left(3\right) & = & 0
\end{array}.
\]
We easily check the validity of the only two possible graceful expansions
of $f$ prescribed with respect to permutation bases $\gamma$ and
$\gamma^{\prime}$ defined for all $t\in\left\{ 0,1\right\} $ and
for all $i\in\mathbb{Z}_{4}$ by 
\[
\sigma_{\gamma}^{\left(-1\right)}\varphi^{\left(t\right)}\left(\varphi^{\left(t\right)}\sigma_{\gamma}\left(i\right)+\left(-1\right)^{t}\cdot\mathfrak{s}_{f}\left(\gamma,\sigma_{\gamma}\left(i\right)\right)\cdot\gamma\sigma_{\gamma}\left(i\right)\right)=f\left(i\right)=\sigma_{\gamma^{\prime}}^{\left(-1\right)}\varphi^{\left(t\right)}\left(\varphi^{\left(t\right)}\sigma_{\gamma^{\prime}}\left(i\right)+\left(-1\right)^{t}\cdot\mathfrak{s}_{f}\left(\gamma^{\prime},\sigma_{\gamma^{\prime}}\left(i\right)\right)\cdot\gamma^{\prime}\sigma_{\gamma^{\prime}}\left(i\right)\right).
\]
\end{example}

\begin{defn}
Functional directed graphs $G_{f}$, $G_{g}$ of $f,g\in\mathbb{Z}_{n}^{\mathbb{Z}_{n}}$,
differ from one another by\emph{ fixed point swaps} if
\[
\left\{ \left\{ i,f\left(i\right)\right\} :\begin{array}{c}
i\in\mathbb{Z}_{n}\\
i\ne f\left(i\right)
\end{array}\right\} =\left\{ \left\{ i,g\left(i\right)\right\} :\begin{array}{c}
i\in\mathbb{Z}_{n}\\
i\ne g\left(i\right)
\end{array}\right\} .
\]
\end{defn}

Let non--isomorphic graphs $G_{f}$ and $G_{g}$ be both connected
and graceful. If $G_{f}$ differs from $G_{g}$ by swapping fixed
points, then we devise from distinct graceful expansions of $g$ distinct
graceful expansions of $f$ and vice versa. Incidentally, the set
of permutation bases for graceful expansions of $f$ bijectively maps
onto the set of permutation bases for graceful expansions of $g$.
In particular, every graceful relabeling of a connected graceful graph
$G_{f}$ admits a unique gracefully labeled swapped fixed point counterpart
whose loop edge is located at the vertex labeled $0$. Consequently,
to characterize permutation bases of functions whose functional digraphs
have no single vertex component, it suffices to characterize permutation
bases which fix $0$. It is easy to see that a permutation $\gamma$
subject to $\gamma\left(0\right)=0$, is a permutation basis for some
graceful expansion if and only if 
\begin{equation}
\forall\,i\in\mathbb{Z}_{n}\backslash\left\{ 0\right\} ,\ \left(\left\{ i-\gamma\left(i\right),i+\gamma\left(i\right)\right\} \cap\mathbb{Z}_{n}\right)\ne\emptyset\Leftrightarrow\begin{cases}
\begin{array}{c}
\gamma\left(i\right)\le i\\
\mbox{ or }\\
\gamma\left(i\right)\le\left(n-1\right)-i
\end{array}\forall\,i\in\mathbb{Z}_{n}\end{cases}.\label{permutation_basis}
\end{equation}
Note that given such a basis $\gamma$ it is possible that for an
input $i\in\mathbb{Z}_{n}$, both conditions 
\[
\gamma\left(i\right)\le i\:\text{ and }\:\gamma\left(i\right)\le\left(n-1\right)-i
\]
simultaneously hold. In that case $\gamma$ is a permutation basis
for two or more functions in $\mathbb{Z}_{n}^{\mathbb{Z}_{n}}$. We
now state and prove a result which characterizes permutation bases
of graceful expansions for members of $\mathbb{Z}_{n}^{\mathbb{Z}_{n}}$
whose graphs have no single vertex component when $n>2$.
\begin{thm}
There are exactly $\left(\left\lfloor \frac{n-1}{2}\right\rfloor !\right)\left(\left\lceil \frac{n-1}{2}\right\rceil !\right)$
distinct permutations which fix $0$, and occur as permutation bases
for graceful expansions of members of $\mathbb{Z}_{n}^{\mathbb{Z}_{n}}$.
\end{thm}

The proof of this claim follows as a corollary of theorem (\ref{Graceful_Expansion_Theorem}).
Let $f\in\mathbb{Z}_{n}^{\mathbb{Z}_{n}}$ be subject to $f\left(0\right)=0$
and such that $G_{f}$ is already gracefully labeled thereby simplifying
the graceful expansion to the setting where $\sigma_{\gamma}=$ id
(the identity permutation). Our argument focuses on the row addition
setup described below where $t\in\left\{ 0,1\right\} $. The setup
stems from the graceful expansion theorem\\

\begin{center}
\begin{tabular}{rr@{\extracolsep{0pt}.}lr@{\extracolsep{0pt}.}lr@{\extracolsep{0pt}.}lr@{\extracolsep{0pt}.}lr@{\extracolsep{0pt}.}lr@{\extracolsep{0pt}.}lr@{\extracolsep{0pt}.}l}
 & \multicolumn{2}{c}{$\varphi^{\left(t\right)}\left(0\right)$} & \multicolumn{2}{c}{$\cdots$} & \multicolumn{2}{c}{$\varphi^{\left(t\right)}\left(i\right)$} & \multicolumn{2}{c}{$\cdots$} & \multicolumn{2}{c}{$\varphi^{\left(t\right)}\left(n-1\right)$} & \multicolumn{2}{c}{} & \multicolumn{2}{c}{Row 1}\tabularnewline
$+$ & \multicolumn{2}{c}{} & \multicolumn{2}{c}{} & \multicolumn{2}{c}{} & \multicolumn{2}{c}{} & \multicolumn{2}{c}{} & \multicolumn{2}{c}{} & \multicolumn{2}{c}{}\tabularnewline
\multirow{1}{*}{} & \multicolumn{2}{c}{\multirow{1}{*}{$\left(-1\right)^{t}\cdot\mathfrak{s}\left(\gamma,0\right)\cdot\gamma\left(0\right)$}} & \multicolumn{2}{c}{$\cdots$} & \multicolumn{2}{c}{$\left(-1\right)^{t}\cdot\mathfrak{s}\left(\gamma,i\right)\cdot\gamma\left(i\right)$} & \multicolumn{2}{c}{$\cdots$} & \multicolumn{2}{c}{$\left(-1\right)^{t}\cdot\mathfrak{s}\left(\gamma,n-1\right)\cdot\gamma\left(n-1\right)$} & \multicolumn{2}{c}{} & \multicolumn{2}{c}{Row 2}\tabularnewline
\midrule
\midrule 
$=$ & \multicolumn{2}{c}{$\varphi^{\left(t\right)}f\left(0\right)$} & \multicolumn{2}{c}{$\cdots$} & \multicolumn{2}{c}{$\varphi^{\left(t\right)}f\left(i\right)$} & \multicolumn{2}{c}{$\cdots$} & \multicolumn{2}{c}{$\varphi^{\left(t\right)}f\left(n-1\right)$} & \multicolumn{2}{c}{} & \multicolumn{2}{c}{Row 1 + Row 2}\tabularnewline
 & \multicolumn{2}{c}{} & \multicolumn{2}{c}{} & \multicolumn{2}{c}{} & \multicolumn{2}{c}{} & \multicolumn{2}{c}{} & \multicolumn{2}{c}{} & \multicolumn{2}{c}{}\tabularnewline
\end{tabular}
\par\end{center}
\begin{proof}
For a permutation basis $\gamma$ which fixes $0$, observe that $f\left(0\right)=0\Leftrightarrow\mathfrak{s}\left(\gamma,0\right)=0$.
This implies that $f\left(n-1\right)=0$ and $G_{f}$ has no isolated
vertex component. In a graceful expansion whose permutation basis
$\gamma$ fixes $0$, there is a unique choice for $\mathfrak{s}_{f}\left(\gamma,n-1\right)\cdot\gamma\left(n-1\right)$.
That choice is 
\[
\mathfrak{s}_{f}\left(\gamma,n-1\right)\cdot\gamma\left(n-1\right)=-\left(n-1\right)\Longleftrightarrow f\left(n-1\right)=0.
\]
Following this assignment, there are two mutually exclusive choices
for a column entry on the second row (of the row addition setup) whose
absolute value equals $\left(n-2\right)$. These mutually exclusive
choices yield corresponding mutually exclusive assignments
\begin{equation}
\begin{cases}
\begin{array}{ccccc}
\mathfrak{s}_{f}\left(\gamma,1\right)\cdot\gamma\left(1\right) & = & n-2 & \Longleftrightarrow & f\left(1\right)=n-1\\
 & \text{or}\\
\mathfrak{s}_{f}\left(\gamma,n-2\right)\cdot\gamma\left(n-2\right) & = & -\left(n-2\right) & \Longleftrightarrow & f\left(n-2\right)=0
\end{array}\end{cases}.\label{Choices}
\end{equation}
Following previous assignments (accounting thus far for column entries
on the second row whose magnitudes are respectively $\left(n-1\right)$
and $\left(n-2\right)$), there are three mutually exclusive choices
for the column entry on the second row whose magnitude is $\left(n-3\right)$.
Mutually exclusive choices ( not accounting for the choice made for
the entry whose magnitude is equal to $\left(n-2\right)$ ) yield
mutually exclusive assignments
\begin{equation}
\begin{cases}
\begin{array}{ccccc}
\mathfrak{s}_{f}\left(\gamma,1\right)\cdot\gamma\left(1\right) & = & n-3 & \Longleftrightarrow & f\left(1\right)=n-2\\
 & \text{or}\\
\mathfrak{s}_{f}\left(\gamma,2\right)\cdot\gamma\left(2\right) & = & n-3 & \Longleftrightarrow & f\left(2\right)=n-1\\
 & \text{or}\\
\mathfrak{s}_{f}\left(\gamma,n-3\right)\cdot\gamma\left(n-3\right) & = & -\left(n-3\right) & \Longleftrightarrow & f\left(n-3\right)=0\\
 & \text{or}\\
\mathfrak{s}_{f}\left(\gamma,n-2\right)\cdot\gamma\left(n-2\right) & = & -\left(n-3\right) & \Longleftrightarrow & f\left(n-2\right)=1
\end{array}\end{cases}.\label{Choices2}
\end{equation}
However either the first or the last assignment displayed in Eq. (\ref{Choices2})
is not possible per the previous choice made for the column entry
of the second row whose magnitude equals $\left(n-2\right)$ as described
in Eq. (\ref{Choices}). Therefore there are three mutually exclusive
choices left for the column entry in the second row whose magnitude
equals $\left(n-3\right)$. Similarly, following the third assignment
(accounting thus far for the three choices made for column entries
of the second row of magnitudes $\left(n-1\right)$, $\left(n-2\right)$
and $\left(n-3\right)$), there are four mutually exclusive choices
for a column entry of the second row whose magnitude equals $\left(n-4\right)$.
Mutually exclusive choices for the entry of the second row with magnitude
$\left(n-4\right)$ (not accounting for the two previously made choices
for entries of the second row having magnitudes $\left(n-2\right)$
and $\left(n-3\right)$ ) yield mutually exclusive assignments 
\begin{equation}
\begin{cases}
\begin{array}{ccccc}
\mathfrak{s}_{f}\left(\gamma,1\right)\cdot\gamma\left(1\right) & = & n-4 & \Longleftrightarrow & f\left(1\right)=n-3\\
 & \text{or}\\
\mathfrak{s}_{f}\left(\gamma,2\right)\cdot\gamma\left(2\right) & = & n-4 & \Longleftrightarrow & f\left(2\right)=n-2\\
 & \text{or}\\
\mathfrak{s}_{f}\left(\gamma,3\right)\cdot\gamma\left(3\right) & = & n-4 & \Longleftrightarrow & f\left(3\right)=n-1\\
 & \text{or}\\
\mathfrak{s}_{f}\left(\gamma,n-4\right)\cdot\gamma\left(n-4\right) & = & -\left(n-4\right) & \Longleftrightarrow & f\left(n-4\right)=0\\
 & \text{or}\\
\mathfrak{s}_{f}\left(\gamma,n-3\right)\cdot\gamma\left(n-3\right) & = & -\left(n-4\right) & \Longleftrightarrow & f\left(n-3\right)=1\\
 & \text{or}\\
\mathfrak{s}_{f}\left(\gamma,n-2\right)\cdot\gamma\left(n-2\right) & = & -\left(n-4\right) & \Longleftrightarrow & f\left(n-2\right)=2
\end{array}\end{cases}.\label{choice3}
\end{equation}
Two of the six possible assignments described in equation (\ref{choice3})
are not possible per the previous assignments made for column entries
of magnitudes $\left(n-2\right)$ and $\left(n-3\right)$ as described
by equation (\ref{Choices}) and equation (\ref{Choices2}). Therefore
there are four mutually exclusive choices left for the column entry
in the second row whose magnitude equals $\left(n-4\right)$. The
argument proceeds accordingly in a similar vein all the way down to
mutually exclusive choices for the column entry in the second row
whose magnitude equals $\left\lceil \frac{n-1}{2}\right\rceil $.
These choices for the partial assignment accounts for the $\left\lfloor \frac{n-1}{2}\right\rfloor !$
factor in the claim. Note that for each choice made in this partial
assignment, thus far the corresponding output of the sign function
is uniquely determined. The remaining unassigned integers whose magnitudes
ranges from $1$ to $\left\lfloor \frac{n-1}{2}\right\rfloor $ can
be arbitrarily permuted among the remaining unassigned column entries
of the second row. Thus accounting for the remaining $\left\lceil \frac{n-1}{2}\right\rceil !$
factor in the claim.
\end{proof}
The proof argument establishes that a permutation $\gamma$ which
fixes $0$, can be a permutation basis for at most $2^{\left\lceil \frac{n-1}{2}\right\rceil }$
distinct members of $\mathbb{Z}_{n}^{\mathbb{Z}_{n}}$ whose graphs
are gracefully labeled. This upper bound is sharp when the permutation
basis is set to $\gamma=$ id.
\begin{defn}
GrL$\left(G_{f}\right)$ denotes the largest subset of distinct gracefully
labeled functional directed graphs isomorphic to $G_{f}$. More formally
we write
\[
\text{GrL}\left(G_{f}\right)\,:=\left\{ G_{\theta f\theta^{-1}}:\begin{array}{c}
\theta\text{ is a representative of a coset in }\nicefrac{\text{S}_{n}}{\text{Aut}\left(G_{f}\right)}\text{ and }\\
\mathbb{Z}_{n}=\left\{ \left|\theta f\theta^{-1}\left(i\right)-i\right|\,:\,i\in\mathbb{Z}_{n}\right\} 
\end{array}\right\} 
\]
\end{defn}

Theorem (\ref{Graceful_Expansion_Theorem}) yields a toy model illustration
of the\emph{ composition lemma}. We discuss the composition lemma
in more detail shortly. For now it suffices to say that the idea of
the composition lemma is to relate graceful expansions of $f$ to
graceful expansions of some non--trivial iterate of $f$. As illustration,
consider the setting where $G_{f}$ is graceful and $f\in$ S$_{n}\subset\mathbb{Z}_{n}^{\mathbb{Z}_{n}}$
i.e. $f$ is bijective.
\begin{prop}
\label{Cycle_Orrientation_Proposition}Let $f\in$ S$_{n}$ and let
$o_{f}$ denote the order of $f$ i.e. the LCM of cycle lengths occurring
in $G_{f}$. The iterate $f^{\left(\text{o}_{f}-1\right)}$ admits
a graceful expansion if and only if the original function $f$ admits
a graceful expansion.
\end{prop}

\begin{proof}
We proceed from the premise that $f$ admits a graceful expansion
and derive via a sequence of reversible steps a graceful expansion
for $f^{\left(\text{o}_{f}-1\right)}$. Assume without loss of generality
that $G_{f}$ is gracefully labeled thereby simplifying the graceful
expansion to the setting where $\sigma_{\gamma}=$ id. By theorem
\ref{Graceful_Expansion_Theorem}, the graceful expansion of $f$
is of the form
\[
f\left(i\right)=\varphi^{\left(t\right)}\left(\varphi^{\left(t\right)}\left(i\right)+\left(-1\right)^{t}\cdot\mathfrak{s}_{f}\left(\gamma,i\right)\cdot\gamma\left(i\right)\right),\:\forall\:i\in\mathbb{Z}_{n}\text{ and }t\in\left\{ 0,1\right\} .
\]
\[
\begin{array}{cccc}
\Longleftrightarrow & \varphi^{\left(t\right)}\left(i\right) & = & \varphi^{\left(t\right)}f\left(i\right)+\left(-1\right)^{t+1}\cdot\mathfrak{s}_{f}\left(\gamma,i\right)\cdot\gamma\left(i\right),\\
\\
\Longleftrightarrow & \varphi^{\left(t\right)}f^{\left(o_{f}-1\right)}\left(i\right) & = & \varphi^{\left(t\right)}\left(i\right)+\left(-1\right)^{t+1}\cdot\mathfrak{s}_{f}\left(\gamma,f^{\left(o_{f}-1\right)}\left(i\right)\right)\cdot\gamma f^{\left(o_{f}-1\right)}\left(i\right),\\
\\
\Longleftrightarrow & f^{\left(o_{f}-1\right)}\left(i\right) & = & \varphi^{\left(t\right)}\left(\varphi^{\left(t\right)}\left(i\right)+\left(-1\right)^{t+1}\cdot\mathfrak{s}_{f}\left(\gamma,f^{\left(o_{f}-1\right)}\left(i\right)\right)\cdot\gamma f^{\left(o_{f}-1\right)}\left(i\right)\right),\\
\\
\Longleftrightarrow & f^{\left(o_{f}-1\right)}\left(i\right) & = & \varphi^{\left(t\right)}\left(\varphi^{\left(t\right)}\left(i\right)+\left(-1\right)^{t}\cdot\mathfrak{s}_{f^{\left(o_{f}-1\right)}}\left(\gamma f^{\left(o_{f}-1\right)},i\right)\cdot\gamma f^{\left(o_{f}-1\right)}\left(i\right)\right),
\end{array}
\]
where 
\[
\mathfrak{s}_{f^{\left(o_{f}-1\right)}}\left(\gamma f^{\left(o_{f}-1\right)},i\right)=-\mathfrak{s}_{f}\left(\gamma,f^{\left(o_{f}-1\right)}\left(i\right)\right),\ \forall\:i\in\mathbb{Z}_{n}.
\]
This completes the proof.
\end{proof}
Accounting for the complementary labeling symmetry described in equation
(\ref{Complementary Symmetry}) and applying the argument used in
the proof of proposition (\ref{Cycle_Orrientation_Proposition}),
to each non--trivial directed cycle occurring in $G_{f}$, it is
easy to see that for all 
\[
f\in\text{S}_{n}\cup\left\{ g:\mathbb{Z}_{n}^{\mathbb{Z}_{n}}:\left|g^{\left(n-1\right)}\left(\mathbb{Z}_{n}\right)\right|=1\right\} ,
\]
$\left|\text{GrL}\left(G_{f}\right)\right|$ is a multiple of $2^{\left(\text{number of connected components in }G_{f}\right)}$.
\begin{example}
\label{Star_Example}The GLC is easily verified for the family of
functional directed graphs of identically constant functions in $\mathbb{Z}_{n}^{\mathbb{Z}_{n}}$
i.e. the family of functional stars. Functional stars are functional
directed graphs of identically constant functions. For instance take
\[
f\,:\,\mathbb{Z}_{n}\rightarrow\mathbb{Z}_{n}
\]
\[
\text{such that}
\]
\[
f\left(i\right)=0,\quad\forall\:i\in\mathbb{Z}_{n}.
\]
We see that the functional directed graph of $f$ is gracefully labeled.
Furthermore for all $n>1$ we have
\[
\text{GrL}\left(G_{f}\right)=\left\{ G_{f},\,G_{\left(n-1-\text{id}\right)f\left(n-1-\text{id}\right)^{-1}}\right\} .
\]
\end{example}

\subsection{Preliminaries}

Recall that univariate polynomial notions such as the LCM and the
GCD do not generally extend to multivariate polynomials. However we
describe special settings where the notion of LCM extends to multivariate
polynomials. Let polynomials $F,\,H\in\mathbb{Q}\left[x_{0},\cdots,x_{n-1}\right]$
split into irreducible multilinear factors
\[
F\left(x_{0},\cdots,x_{n-1}\right)=\prod_{0\le i<m}\left(P_{i}(x_{0},\cdots,x_{n-1})\right)^{\alpha_{i}},\quad H\left(x_{0},\cdots,x_{n-1}\right)=\prod_{0\le i<m}\left(P_{i}(x_{0},\cdots,x_{n-1})\right)^{\beta_{i}}.
\]
In the factorization above assume that $\left\{ \alpha_{i},\beta_{i}\,:\,0\le i<m\right\} \subset\mathbb{Z}_{\ge0}$
and more importantly that each factor $P_{i}\left(\mathbf{x}\right)$
is a $\mathbb{Q}$--linear combination of variables $x_{0},\cdots,x_{n-1}$
of the form
\[
P_{i}\left(x_{0},\cdots,x_{n-1}\right)=\sum_{j\in\mathbb{Z}_{n}}a_{i,j}\,x_{j},\text{ where }\left\{ a_{ij}\,:\,\begin{array}{c}
0\le i<m\\
0\le j<n
\end{array}\right\} \subset\mathbb{Q}.
\]
Additionally, assume for each $k\in\mathbb{Z}_{n}$, and each factor
$P_{i}$ when viewed as a univariate polynomial in $x_{k}$ (with
coefficients from the ring $\mathbb{Q}\left[x_{0},\cdots,x_{k-1},x_{k+1},\cdots,x_{n-1}\right]$)
has no common roots with any other factor in $\left\{ P_{j}:0\le j\ne i<m\right\} $
in the field of fractions $\mathbb{Q}\left(x_{0},\cdots,x_{k-1},x_{k+1},\cdots,x_{n-1}\right)$,
in other words the resultant in the variable $x_{k}$ given by 
\[
\prod_{\begin{array}{c}
0\le u<v<m\\
0\notin\left\{ a_{v,k},a_{u,k}\right\} 
\end{array}}\left(\sum_{t\in\mathbb{Z}_{n}\backslash\left\{ k\right\} }\frac{a_{v,t}}{a_{v,k}}x_{t}-\sum_{s\in\mathbb{Z}_{n}\backslash\left\{ k\right\} }\frac{a_{u,s}}{a_{u,k}}x_{s}\right)
\]
does not vanish identically. In this restricted setting we can extend
the notion of LCM and GCD to multivariate polynomials $F$ and $H$
as follows 
\[
\text{LCM}\left(F,\,H\right)=\prod_{0\le i<m}\left(P_{i}(x_{0},\cdots,x_{n-1})\right)^{\max\left(\alpha_{i},\beta_{i}\right)},
\]
and
\[
\text{GCD}\left(F,\,H\right)=\prod_{0\le i<m}\left(P_{i}(x_{0},\cdots,x_{n-1})\right)^{\min\left(\alpha_{i},\beta_{i}\right)}.
\]
For convenience we adopt the falling factorial notation :
\[
\left(x\right)^{\underline{n}}\,:=\prod_{i\in\mathbb{Z}_{n}}\left(x-i\right).
\]
Also for an arbitrary polynomial $P\left(x_{0},\cdots,x_{m-1}\right)\in\mathbb{Q}\left[x_{0},\cdots,x_{m-1}\right]$
and $g\in\mathbb{Z}_{n}^{\mathbb{Z}_{m}}$, we denote by $P\left(g\right)$
the evaluation of $P$ at the sequence $\left(x_{i}=g\left(i\right):i\in\mathbb{Z}_{m}\right)$
i.e.
\[
P\left(g\right):=P\left(g\left(0\right),\cdots,g\left(i\right),\cdots,g\left(m-1\right)\right).
\]

\begin{prop}
\label{Quotient_Remainder_Expansion} Every $H\in\mathbb{Q}\left[x_{0},\cdots,x_{m-1}\right]$
admits a quotient-remainder expansion of the form
\[
H\left(x_{0},\cdots,x_{m-1}\right)=\sum_{\ell\in\mathbb{Z}_{m}}q_{\ell}\left(x_{0},\cdots,x_{m-1}\right)\left(x_{\ell}\right)^{\underline{n}}+\sum_{g\in\mathbb{Z}_{n}^{\mathbb{Z}_{m}}}H\left(g\right)\prod_{i\in\mathbb{Z}_{m}}\left(\prod_{j_{i}\in\mathbb{Z}_{n}\backslash\left\{ g\left(i\right)\right\} }\frac{x_{i}-j_{i}}{g\left(i\right)-j_{i}}\right),
\]
where $q_{\ell}\left(x_{0},\cdots,x_{m-1}\right)\in\mathbb{Q}\left[x_{0},\cdots,x_{m-1}\right]$
for all $\ell\in\mathbb{Z}_{m}$.
\end{prop}

\begin{proof}
We prove the claim by induction on $m$ (the number of variables).
The claim in the base case $m=1$, is the assertion that for all $n\ge1$,
$H\left(x_{0}\right)\in\mathbb{Q}\left[x_{0}\right]$ admits an expansion
of the form
\[
H\left(x_{0}\right)=q\left(x_{0}\right)\,\left(x_{0}\right)^{\underline{n}}+r\left(x_{0}\right),
\]
where $r\left(x_{0}\right)$ is a polynomial of degree less then $n$
called the remainder. Since the remainder $r\left(x_{0}\right)$ is
of degree at most $\left(n-1\right)$ it is completely determined
via Lagrange interpolation on $n$ distinct evaluation points as follows
\[
H\left(x_{0}\right)=q\left(x_{0}\right)\,\left(x_{0}\right)^{\underline{n}}+\sum_{g\in\mathbb{Z}_{n}^{\mathbb{Z}_{1}}}H\left(g\left(0\right)\right)\,\prod_{j\in\mathbb{Z}_{n}\backslash\left\{ g\left(0\right)\right\} }\left(\frac{x_{0}-j_{0}}{g\left(0\right)-j_{0}}\right).
\]
Therefore, the claim holds in the base case. Note that the same argument,
including Lagrange interpolation, applies to univariate polynomials
whose coefficients lies in a polynomial ring.

For the induction step, assume as induction hypothesis that the claim
holds for all $m$--variate polynomials $F\in\mathbb{Q}\left[x_{0},\cdots,x_{m-1}\right]$
namely assume that
\[
F=\sum_{\ell\in\mathbb{Z}_{m}}t_{\ell}\left(x_{0},\cdots,x_{m-1}\right)\left(x_{\ell}\right)^{\underline{n}}+\sum_{g\in\mathbb{Z}_{n}^{\mathbb{Z}_{m}}}F\left(g\right)\prod_{k\in\mathbb{Z}_{m}}\left(\prod_{j_{k}\in\mathbb{Z}_{n}\backslash\left\{ g\left(k\right)\right\} }\frac{x_{k}-j_{k}}{g\left(k\right)-j_{k}}\right).
\]
We now show that the hypothesis implies that the claim also holds
for all $\left(m+1\right)$--variate polynomials with rational coefficients.
Let $H\in\mathbb{Q}\left[x_{0},\cdots,x_{m}\right]$ be viewed as
a univariate polynomial in $x_{m}$ whose coefficients lie in the
polynomial ring $\mathbb{Q}\left[x_{0},\cdots,x_{m-1}\right]$. Invoking
the Quotient--Remainder Theorem and Lagrange interpolation over this
ring, we have
\[
H\left(x_{0},\cdots,x_{m}\right)=q\left(x_{0},\cdots,x_{m}\right)\,\left(x_{m}\right)^{\underline{n}}+r\left(x_{0},\cdots,x_{m}\right),
\]
where $r\left(x_{0},\cdots,x_{m}\right)\in\big(\mathbb{Q}[x_{0},\cdots,x_{m-1}]\big)[x_{m}]$
is of degree at most $n-1$ in the variable $x_{m}$. We write 
\[
r\left(x_{0},\cdots,x_{m}\right)=\sum_{k\in\mathbb{Z}_{n}}a_{k}\left(x_{0},\cdots,x_{m-1}\right)\,(x_{m})^{k},
\]
to justify that $a_{k}\left(x_{0},\cdots,x_{m-1}\right)\in\mathbb{Q}[x_{0},\cdots,x_{m-1}]$
for all $k\in\mathbb{Z}_{n}$, observe that
\[
\bigg(\text{Vandermonde}\left(\begin{array}{c}
0\\
\vdots\\
u\\
\vdots\\
n-1
\end{array}\right)\bigg)\cdot\left(\begin{array}{c}
a_{0}\left(x_{0},\cdots,x_{m-1}\right)\\
\vdots\\
a_{u}\left(x_{0},\cdots,x_{m-1}\right)\\
\vdots\\
a_{n-1}\left(x_{0},\cdots,x_{m-1}\right)
\end{array}\right)=\left(\begin{array}{c}
H(x_{0},\cdots,x_{m-1},0)\\
\vdots\\
H(x_{0},\cdots,x_{m-1},u)\\
\vdots\\
H(x_{0},\cdots,x_{m-1},n-1)
\end{array}\right),
\]
where
\[
\bigg(\text{Vandermonde}\left(\begin{array}{c}
0\\
\vdots\\
u\\
\vdots\\
n-1
\end{array}\right)\bigg)\left[i,j\right]=i^{j},\ \forall\,0\le i,j<n.
\]
Since
\[
\det\bigg(\text{Vandermonde}\left(\begin{array}{c}
0\\
\vdots\\
u\\
\vdots\\
n-1
\end{array}\right)\bigg)=\bigg(\prod_{k\in\mathbb{Z}_{n}}k!\bigg)\ne0.
\]
Thus the coefficients of the remainder $r\left(x_{0},\cdots,x_{m}\right)$
from the ring $\mathbb{Q}[x_{0},\cdots,x_{m-1}]$ obtained as entries
of the unique solution vector
\[
\left(\begin{array}{c}
a_{0}\left(x_{0},\cdots,x_{m-1}\right)\\
\vdots\\
a_{u}\left(x_{0},\cdots,x_{m-1}\right)\\
\vdots\\
a_{n-1}\left(x_{0},\cdots,x_{m-1}\right)
\end{array}\right)=\bigg(\text{Vandermonde}\left(\begin{array}{c}
0\\
\vdots\\
u\\
\vdots\\
n-1
\end{array}\right)\bigg)^{-1}\cdot\left(\begin{array}{c}
H(x_{0},\cdots,x_{m-1},0)\\
\vdots\\
H(x_{0},\cdots,x_{m-1},u)\\
\vdots\\
H(x_{0},\cdots,x_{m-1},n-1)
\end{array}\right).
\]
Equivalently, via Lagrange interpolation we write
\[
H=q_{m}\left(x_{0},\cdots,x_{m}\right)\left(x_{m}\right)^{\underline{n}}+\sum_{g\left(m\right)\in\mathbb{Z}_{n}}H\left(x_{0},\cdots,x_{m-1},g\left(m\right)\right)\,\prod_{j\in\mathbb{Z}_{n}\backslash\left\{ g\left(m\right)\right\} }\left(\frac{x_{m}-j_{m}}{g\left(m\right)-j_{m}}\right).
\]
Applying the induction hypothesis to $m$--variate polynomials in
$\left\{ H\left(x_{0},\cdots,x_{m-1},g\left(m\right)\right):g\left(m\right)\in\mathbb{Z}_{n}\right\} $
yields the desired claim.
\end{proof}
\begin{defn}
For an arbitrary $H\in\mathbb{Q}\left[x_{0},\cdots,x_{m-1}\right]$,
the \emph{canonical representative} of the congruence class of $H$
modulo the ideal generated by $\left\{ \left(x_{i}\right)^{\underline{n}}:i\in\mathbb{Z}_{m}\right\} $
denoted
\[
H\mod\left\{ \left(x_{i}\right)^{\underline{n}}:i\in\mathbb{Z}_{m}\right\} ,
\]
is the unique polynomial of degree at most $\left(n-1\right)$ in
each variable whose evaluations over the integer lattice $\mathbb{Z}_{n}^{\mathbb{Z}_{m}}$
i.e. $\left(\mathbb{Z}_{n}\right)^{m}$ matches evaluations of $H$
on the same lattice. Thus the canonical representative of 
\[
H\mod\left\{ \left(x_{i}\right)^{\underline{n}}:i\in\mathbb{Z}_{m}\right\} ,
\]
is
\begin{equation}
\sum_{g\in\mathbb{Z}_{n}^{\mathbb{Z}_{m}}}H\left(g\right)\,\prod_{k\in\mathbb{Z}_{m}}\left(\prod_{j_{k}\in\mathbb{Z}_{n}\backslash\left\{ g\left(k\right)\right\} }\frac{x_{k}-j_{k}}{g\left(k\right)-j_{k}}\right).\label{Canonical representative}
\end{equation}
The quotient--divisor part associated with the congruence class 
\[
H\mod\left\{ \left(x_{i}\right)^{\underline{n}}:i\in\mathbb{Z}_{m}\right\} ,
\]
is the polynomial
\[
H-\sum_{g\in\mathbb{Z}_{n}^{\mathbb{Z}_{m}}}H\left(g\right)\,\prod_{k\in\mathbb{Z}_{m}}\left(\prod_{j_{k}\in\mathbb{Z}_{n}\backslash\left\{ g\left(k\right)\right\} }\frac{x_{k}-j_{k}}{g\left(k\right)-j_{k}}\right).
\]
We see that the quotient--divisor part vanishes identically on the
lattice $\mathbb{Z}_{n}^{\mathbb{Z}_{m}}$.
\end{defn}

Thus the canonical representative of $H\in\mathbb{Q}\left[x_{0},\cdots,x_{n-1}\right]$
is obtained via Lagrange interpolation over evaluation points 
\[
\left\{ \left(g,\,H\left(g\right)\right):g\in\mathbb{Z}_{n}^{\mathbb{Z}_{n}}\right\} .
\]
Alternatively the canonical representative is obtained as the final
remainder devised by performing Euclidean divisions irrespective of
the order with which we perform the division by divisors successively
taken from univariate polynomials $\left\{ \left(x_{i}\right)^{\underline{n}}:i\in\mathbb{Z}_{n}\right\} $.
This follows from the fact that generators $\left\{ \left(x_{i}\right)^{\underline{n}}:i\in\mathbb{Z}_{n}\right\} $
for the corresponding ideal form a Groebner basis.

The next result recursively applies at each iteration the Quotient--Remainder
Expansion Theorem (\ref{Quotient_Remainder_Expansion}) to the quotient--divisor
part of the expansion to express the input polynomial as a $\mathbb{Q}$--linear
combination of Lagrange basis polynomials. For an arbitrary $g\in\mathbb{Z}_{n}^{\mathbb{Z}_{m}}$,
let the associated Lagrange basis polynomial be
\[
L_{g}\left(x_{0},\cdots,x_{i},\cdots,x_{m-1}\right):=\prod_{u\in\mathbb{Z}_{m}}\left(\prod_{v_{u}\in\mathbb{Z}_{n}\backslash\left\{ g\left(u\right)\right\} }\frac{x_{u}-v_{u}}{g\left(u\right)-v_{u}}\right).
\]
It follows from the diagonality criterion prescribed for all $\left(f,g\right)\in\mathbb{Z}_{n}^{\mathbb{Z}_{n}}\times\mathbb{Z}_{n}^{\mathbb{Z}_{n}}$
by
\[
L_{f}\left(f\right)=1\,\text{ and }\,L_{f}\left(g\right)=0\,\text{ when }f\ne g,
\]
 that Lagrange basis polynomials $\left\{ L_{f}\left(x_{0},\cdots,x_{n-1}\right):f\in\mathbb{Z}_{n}^{\mathbb{Z}_{n}}\right\} $
are linearly independent.
\begin{cor}
\label{Quotient_Remainder_Corollary} Let $P\in\mathbb{Q}\left[x_{0},\cdots,x_{n-1}\right]$
be a polynomial of degree at most $d\ge n$ in any of its variables
$x_{0},\cdots,x_{n-1}$ then
\[
P\left(x_{0},\cdots,x_{n-1}\right)=\sum_{0\le k\le d+1-n}\left(\sum_{g_{k}\in\mathbb{Z}_{n+k}^{\mathbb{Z}_{n}}}b_{g_{k}}\,L_{g_{k}}\left(x_{0},\cdots,x_{n-1}\right)\right),
\]
where for all $t<d-n$,
\[
b_{g_{0}}=P\left(g_{0}\right)\:\text{ and }\:b_{g_{t+1}}=P\left(g_{t+1}\right)-\sum_{0\le k\le t}\left(\sum_{g_{k}\in\mathbb{Z}_{n+k}^{\mathbb{Z}_{n}}}b_{g_{k}}\,L_{g_{k}}\left(g_{t+1}\right)\right).
\]
\end{cor}

\begin{proof}
Applying proposition \ref{Quotient_Remainder_Expansion} to $P$ yields
an expansion of the form
\[
P\left(x_{0},\cdots,x_{n-1}\right)=\sum_{\ell\in\mathbb{Z}_{n}}q_{0,\ell}\left(x_{0},\cdots,x_{n-1}\right)\left(x_{\ell}\right)^{\underline{n}}+\sum_{g_{0}\in\mathbb{Z}_{n+0}^{\mathbb{Z}_{n}}}P\left(g_{0}\right)\,L_{g_{0}}\left(x_{0},\cdots,x_{n-1}\right).
\]
We set $b_{g_{0}}=P\left(g_{0}\right)$ for all $g_{0}\in\mathbb{Z}_{n+0}^{\mathbb{Z}_{n}}$
and we write
\[
P\left(x_{0},\cdots,x_{n-1}\right)=\sum_{\ell\in\mathbb{Z}_{n}}q_{0,\ell}\left(x_{0},\cdots,x_{n-1}\right)\left(x_{\ell}\right)^{\underline{n}}+\sum_{g_{0}\in\mathbb{Z}_{n+0}^{\mathbb{Z}_{n}}}b_{g_{0}}\,L_{g_{0}}\left(x_{0},\cdots,x_{n-1}\right).
\]
The quotient divisor part 
\[
\sum_{\ell\in\mathbb{Z}_{n}}q_{0,\ell}\left(x_{0},\cdots,x_{n-1}\right)\left(x_{\ell}\right)^{\underline{n}}=P\left(x_{0},\cdots,x_{n-1}\right)-\sum_{g_{0}\in\mathbb{Z}_{n+0}^{\mathbb{Z}_{n}}}b_{g_{0}}\,L_{g_{0}}\left(x_{0},\cdots,x_{n-1}\right).
\]
Applying proposition. \ref{Quotient_Remainder_Expansion} to the quotient
remainder part, yields an expansion of $P$ the form
\[
P\left(x_{0},\cdots,x_{n-1}\right)=
\]
\[
\sum_{\ell\in\mathbb{Z}_{n}}q_{1,\ell}\left(x_{0},\cdots,x_{n-1}\right)\left(x_{\ell}\right)^{\underline{n+1}}+\sum_{g_{1}\in\mathbb{Z}_{n+1}^{\mathbb{Z}_{n}}}b_{g_{1}}\,L_{g_{1}}\left(x_{0},\cdots,x_{n-1}\right)+\sum_{g_{0}\in\mathbb{Z}_{n+0}^{\mathbb{Z}_{n}}}P\left(g_{0}\right)\,L_{g_{0}}\left(x_{0},\cdots,x_{n-1}\right),
\]
where
\[
b_{g_{1}}=P\left(g_{1}\right)-\sum_{g_{0}\in\mathbb{Z}_{n+0}^{\mathbb{Z}_{n}}}b_{g_{0}}\,L_{g_{0}}\left(g_{1}\right).
\]
Continuing in a similar vein we iterate proposition \ref{Quotient_Remainder_Expansion}
on subsequent quotient--divisor parts. The procedure terminates at
the iteration where the quotient--divisor part vanishes identically.
Thus completing the proof.
\end{proof}

\subsection{Determinantal Certificate}

The following proposition expresses via a determinantal certificate
a second necessary and sufficient condition for the functional directed
graph $G_{f}$ of $f\in\mathbb{Z}_{n}^{\mathbb{Z}_{n}}$ to be graceful.
\begin{prop}
\label{Certificate_of_Grace} The directed graph $G_{f}$ of $f\in\mathbb{Z}_{n}^{\mathbb{Z}_{n}}$
is graceful if and only if
\[
0\not\equiv\text{LCM}\left(\prod_{0\le i<j<n}\left(x_{j}-x_{i}\right),\,\prod_{0\le u<v<n}\left(\left(x_{f\left(v\right)}-x_{v}\right)^{2}-\left(x_{f\left(u\right)}-x_{u}\right)^{2}\right)\right)\text{mod}\left\{ \left(x_{k}\right)^{\underline{n}}:k\in\mathbb{Z}_{n}\right\} 
\]
\end{prop}

\begin{proof}
The LCM in the assertion is well defined since
\[
\prod_{0\le u<v<n}\left(\left(x_{f\left(v\right)}-x_{v}\right)^{2}-\left(x_{f\left(u\right)}-x_{u}\right)^{2}\right)=\prod_{0\le u<v<n}\left(x_{f\left(v\right)}-x_{v}-x_{f\left(u\right)}+x_{u}\right)\left(x_{f\left(v\right)}-x_{v}+x_{f\left(u\right)}-x_{u}\right).
\]
We see that the second input to the LCM corresponds to a product of
$\left\{ -2,-1,0,1,2\right\} $--linear combination of variables.
The first Vandermonde determinant factor
\[
\prod_{0\le i<j<n}\left(x_{j}-x_{i}\right)
\]
vanishes whenever two distinct variables are assigned the same vertex
label from $\mathbb{Z}_{n}$. Similarly, the second Vandermonde determinant
factor 
\[
\prod_{0\le u<v<n}\left(\left(x_{f\left(v\right)}-x_{v}\right)^{2}-\left(x_{f\left(u\right)}-x_{u}\right)^{2}\right)
\]
vanishes whenever two distinct edges are assigned the same induced
absolute subtractive edge label from $\mathbb{Z}_{n}$. Consider the
following multiple of the LCM polynomial
\[
F_{f}\left(x_{0},\cdots,x_{n-1}\right)=\prod_{0\le i<j<n}\left(x_{j}-x_{i}\right)\left(\left(x_{f\left(j\right)}-x_{j}\right)^{2}-\left(x_{f\left(i\right)}-x_{i}\right)^{2}\right).
\]
It suffices to show that the functional directed graph $G_{f}$ of
$f\in\mathbb{Z}_{n}^{\mathbb{Z}_{n}}$ is graceful if and only if
\[
0\not\equiv F_{f}\left(x_{0},\cdots,x_{n-1}\right)\text{ mod}\left\{ \left(x_{k}\right)^{\underline{n}}:k\in\mathbb{Z}_{n}\right\} .
\]
 By proposition (\ref{Quotient_Remainder_Expansion}), the polynomial
$F_{f}$ admits an expansion of the form
\[
F_{f}(x_{0},\cdots,x_{n-1})=\sum_{\ell\in\mathbb{Z}_{n}}q_{\ell}\left(x_{0},\cdots,x_{n-1}\right)\left(x_{\ell}\right)^{\underline{n}}+\sum_{g\in\mathbb{Z}_{n}^{\mathbb{Z}_{n}}}F_{f}\left(g\right)L_{g}\left(x_{0},\cdots,x_{n-1}\right).
\]
Note that for all $g\in\mathbb{Z}_{n}^{\mathbb{Z}_{n}}$
\[
F_{f}\left(g\right)=\prod_{0\le i<j<n}\left(g\left(j\right)-g\left(i\right)\right)\left(\left(gf\left(j\right)-g\left(j\right)\right)^{2}-\left(gf\left(i\right)-g\left(i\right)\right)^{2}\right).
\]
Hence $F_{f}\left(g\right)$ vanishes if $g\in\mathbb{Z}_{n}^{\mathbb{Z}_{n}}\backslash\text{S}_{n}$
on the other hand if $\sigma\in\text{S}_{n}$ and $G_{\sigma f\sigma^{-1}}$
is not a gracefully labeled then $F_{f}\left(\sigma\right)$ also
vanishes. Thus
\[
F_{f}\left(x_{0},\cdots,x_{n-1}\right)=\sum_{\ell\in\mathbb{Z}_{n}}q_{\ell}\left(x_{0},\cdots,x_{n-1}\right)\left(x_{\ell}\right)^{\underline{n}}+
\]
\[
\sum_{\begin{array}{c}
\sigma\in\text{S}_{n}\\
G_{\sigma f\sigma^{-1}}\in\text{GrL}\left(G_{f}\right)
\end{array}}\prod_{0\le i<j<n}\left(\sigma\left(j\right)-\sigma\left(i\right)\right)\left(\left(\sigma f\left(j\right)-\sigma\left(j\right)\right)^{2}-\left(\sigma f\left(i\right)-\sigma\left(i\right)\right)^{2}\right)L_{\sigma}\left(x_{0},\cdots,x_{n-1}\right).
\]
Observe that for all $\gamma\in\left\{ \sigma\in\text{S}_{n}:G_{\sigma f\sigma^{-1}}\in\text{GrL}\left(G_{f}\right)\right\} $
we have
\[
\prod_{0\le i<j<n}\left(\gamma\left(j\right)-\gamma\left(i\right)\right)\left(\left(\gamma f\left(j\right)-\gamma\left(j\right)\right)^{2}-\left(\gamma f\left(i\right)-\gamma\left(i\right)\right)^{2}\right)\in\left\{ -\prod_{v\in\mathbb{Z}_{n}}\left(\left(v!\right)^{2}\frac{\left(n-1+v\right)!}{\left(2v\right)!}\right),\prod_{v\in\mathbb{Z}_{n}}\left(\left(v!\right)^{2}\frac{\left(n-1+v\right)!}{\left(2v\right)!}\right)\right\} .
\]
We thus conclude that 
\[
0\not\equiv F_{f}\left(x_{0},\cdots,x_{n-1}\right)\text{ mod}\left\{ \left(x_{k}\right)^{\underline{n}}:k\in\mathbb{Z}_{n}\right\} .
\]
if and only if $\emptyset\ne\text{GrL}\left(G_{f}\right)$ as claimed.
\end{proof}
We now briefly explain why the proposed polynomial construction is
determinental. Let $\mathbf{V},\mathbf{G}_{f}\in\left(\mathbb{Q}\left[x_{0},\cdots,x_{n-1}\right]\right)^{n\times n}$
with entries given by 
\[
\mathbf{V}\left[i,j\right]=\left(x_{i}\right)^{j},\quad\mathbf{G}_{f}\left[i,j\right]=\left(x_{f\left(j\right)}-x_{j}\right)^{2i}\text{ for all }\:0\le i,j<n,
\]
The matrices $\mathbf{F}$ and $\mathbf{G}_{f}$ are Vandermonde matrices
whose determinants are well known and are respectively 
\[
\det\left(\mathbf{V}\right)=\prod_{0\le i<j<n}\left(x_{j}-x_{i}\right),
\]
and
\[
\det\left(\mathbf{G}_{f}\right)=\prod_{0\le i<j<n}\left(\left(x_{f\left(j\right)}-x_{j}\right)^{2}-\left(x_{f\left(i\right)}-x_{i}\right)^{2}\right).
\]
By the multiplicative property of the determinant, we have
\[
\det\left(\mathbf{V}\right)\det\left(\mathbf{G}_{f}\right)=\det\left(\mathbf{V}\mathbf{G}_{f}\right).
\]
The entries of the matrix product $\mathbf{V}\mathbf{G}_{f}$ are
such that
\[
\left(\mathbf{V}\mathbf{G}_{f}\right)\left[i,j\right]=\frac{1-\left(x_{i}\left(x_{f\left(j\right)}-x_{j}\right)^{2}\right)^{n}}{1-x_{i}\left(x_{f\left(j\right)}-x_{j}\right)^{2}},\,\forall\,0\le i,j<n
\]
and the polynomial of interest is 
\[
\det\left(\mathbf{V}\mathbf{G}_{f}\right)=F_{f}\left(x_{0},\cdots,x_{n-1}\right)=\prod_{0\le i<j<n}\left(x_{j}-x_{i}\right)\left(\left(x_{f\left(j\right)}-x_{j}\right)^{2}-\left(x_{f\left(i\right)}-x_{i}\right)^{2}\right).
\]
We make redundant factors more apparent by factoring det$\left(\mathbf{V}\mathbf{G}_{f}\right)$
into products of linear combinations of the variables as follows
\[
\det\left(\mathbf{V}\mathbf{G}_{f}\right)=\prod_{0\le i<j<n}\left(x_{j}-x_{i}\right)\left(x_{f\left(j\right)}-x_{j}-x_{f\left(i\right)}+x_{i}\right)\left(x_{f\left(j\right)}-x_{j}+x_{f\left(i\right)}-x_{i}\right).
\]
We invoke the LCM as a means of removing from det$\left(\mathbf{V}\mathbf{G}_{f}\right)$
redundant factors from the determinantal construction. For instance
when $G_{f}$ is connected and contains no cycle other then the trivial
cycle made by a loop edge and $f\left(u\right)\le u$ for all $u\in\mathbb{Z}_{n}$,
then redundant factors appear only when
\[
\begin{array}{c}
f\left(i\right)=f\left(j\right)\\
i<j
\end{array}\:\text{ or }\:d\left(i,f^{\left(2\right)}\left(i\right)\right)=2.
\]
where $d\left(u,v\right)$ denotes the non--loop edge distance separating
vertex $u$ from vertex $v$ in $G_{f}$.
\[
\text{LCM}\left(\det\left(\mathbf{V}\right),\,\det\left(\mathbf{G}_{f}\right)\right)=\text{LCM}\left(\prod_{0\le i<j<n}\left(x_{j}-x_{i}\right),\,\prod_{0<i<j<n}\left(\left(x_{f\left(j\right)}-x_{j}\right)^{2}-\left(x_{f\left(i\right)}-x_{i}\right)^{2}\right)\right)=
\]
\[
\prod_{0\le i<j<n}\left(x_{j}-x_{i}\right)\prod_{d\left(i,f^{\left(2\right)}\left(i\right)\right)=2}\left(2x_{f\left(i\right)}-x_{i}-x_{f^{\left(2\right)}\left(i\right)}\right)\prod_{\begin{array}{c}
0\le i<j<n\\
f\left(i\right)=f\left(j\right)
\end{array}}\left(2x_{f\left(j\right)}-x_{j}-x_{i}\right)\times
\]
\[
\prod_{d\left(i,f^{\left(3\right)}\left(i\right)\right)=3}\left(\left(x_{f\left(i\right)}-x_{i}\right)^{2}-\left(x_{f^{\left(3\right)}\left(i\right)}-x_{f^{\left(2\right)}\left(i\right)}\right)^{2}\right)\prod_{\begin{array}{c}
0<i<j<n\\
d\left(i,j\right)\ge3
\end{array}}\left(\left(x_{f\left(j\right)}-x_{j}\right)^{2}-\left(x_{f\left(i\right)}-x_{i}\right)^{2}\right).
\]
Finally note that the canonical representative of $\det\left(\mathbf{V}\mathbf{G}_{f}\right)^{2}$
is
\[
\prod_{v\in\mathbb{Z}_{n}}\left(\left(v!\right)^{2}\,\frac{\left(n-1+v\right)!}{\left(2v\right)!}\right)^{2}\sum_{\begin{array}{c}
\sigma\in\text{S}_{n}\\
G_{\sigma f\sigma^{-1}}\in\text{GrL}\left(G_{f}\right)
\end{array}}\prod_{\begin{array}{c}
i\in\mathbb{Z}_{n}\\
j_{i}\in\mathbb{Z}_{n}\backslash\left\{ \sigma\left(i\right)\right\} 
\end{array}}\left(\frac{x_{i}-j_{i}}{\sigma\left(i\right)-j_{i}}\right).
\]

\section{The transformation monoid $\mathbb{Z}_{n}^{\mathbb{Z}_{n}}$ and
the composition lemma.}

We briefly review basic properties of the transformation monoid $\mathbb{Z}_{n}^{\mathbb{Z}_{n}}$
relevant to our main result.
\begin{prop}
For all $f\in\mathbb{Z}_{n}^{\mathbb{Z}_{n}}$, we have
\[
1\le\min_{\sigma\in\text{S}_{n}}\left|\left\{ \left|\sigma f\sigma^{\left(-1\right)}\left(i\right)-i\right|:i\in\mathbb{Z}_{n}\right\} \right|\le\rho_{f}+\begin{cases}
\begin{array}{cc}
1 & \text{ if }f\left(i\right)=i,\ \forall\,i\in\mathbb{Z}_{n},\\
2 & \text{ else if }\exists\:i\in\mathbb{Z}_{n}\text{ s.t. }f\left(i\right)=i\\
0 & \text{otherwise}
\end{array},\end{cases}
\]
where $\rho_{f}$ denotes the minimum number of non--loop edge deletions
required in $G_{f}$ to obtain a subgraph which is a union of disjoint
paths and two--cycles possibly having loop edges.
\end{prop}

\begin{proof}
The lower bound is attained when the $n$ edges of $G_{f}$ are assigned
the same induced absolute subtractive edge label. We justify the upper--bound
by considering any one of the possible deletions of $\rho_{f}$ edges
from $G_{f}$ to obtain a spanning union of disjoint paths. We then
sequentially label vertices along each path starting from one endpoint
and increasing by one the label for each vertex encountered as we
move along the path towards the second endpoint. This procedure greedily
maximizes the number of edges assigned the induced absolute subtractive
edge label one. The only edges in the proposed relabeling of $G_{f}$
whose induced absolute subtractive edge labels possibly differ from
one, are labels of the $\rho_{f}$ deleted non--loop edges as well
as loop edges if any occurs in $G_{f}$.
\end{proof}
Note that the lower bound is sharp when $f\in$ S$_{n}$ is either
the identity element or any other involution having no fixed points.
The upper bound is sharp for any involution, or alternatively a functional
directed graph made up of a single spanning directed cycle or alternatively
a functional path.
\begin{prop}
For all $f\in\mathbb{Z}_{n}^{\mathbb{Z}_{n}}$,
\[
n\ge\max_{\sigma\in\text{S}_{n}}\left|\left\{ \left|\sigma f\sigma^{-1}\left(i\right)-i\right|:i\in\mathbb{Z}_{n}\right\} \right|\ge\left|\left\{ \left(i,f(i)\right):\begin{array}{c}
i\in\mathbb{Z}_{n}\\
i\ne f(i)
\end{array}\right\} \right|-\rho_{f}+\begin{cases}
\begin{array}{cc}
1 & \text{ if }\exists\:i\in\mathbb{Z}_{n}\text{ s.t. }f(i)=i\\
0 & \text{otherwise}
\end{array}.\end{cases}
\]
where $\rho_{f}$ denotes the minimum number of non--loop edge deletions
required in $G_{f}$ to obtain a subgraph which is a union of disjoint
paths possibly having loop edges.
\end{prop}

\begin{proof}
The upper bound follows from the observation $\left|E\left(G_{f}\right)\right|=n$.
We justify the lower--bound by considering every possible deletion
of $\rho_{f}$ edges from $G_{f}$ to obtain a union of disjoint paths,
possibly having loop edges. We then sequentially label vertices along
each path starting from one endpoint and alternating between largest
and smallest unassigned label for each vertex encountered as we move
along the path towards the second endpoint. This procedure greedily
maximizes the number of distinct edge labels. The only edges in the
proposed relabeling of $G_{f}$ whose induced absolute subtractive
edge labels possibly repeat correspond to labels of deleted non--loop
edges as well as duplicate loop edges if any occurs in $G_{f}$.
\end{proof}
Note that the upper bound is sharp for a graceful functional directed
graph. The lower bound is sharp for any involution or alternatively
a functional directed graph made of a single spanning directed cycle
i.e. a directed Hamiltonian cycle.
\begin{defn}
Let $P\in\mathbb{Q}\left[x_{0},\cdots,x_{n-1}\right]$, we denote
by Aut$\left(P\right)$ the stabilizer subgroup of S$_{n}\subset\mathbb{Z}_{n}^{\mathbb{Z}_{n}}$
associated with the polynomial $P$ defined with respect to permutations
of variables $x_{0},\cdots,x_{n-1}$. In other words Aut$\left(P\right)$
denotes the set of permutations of the variables which fixes $P$.
\end{defn}

\begin{prop}[Stabilizer subgroup]
\label{Stabiliser_Subgroup} For an arbitrary $f\in\mathbb{Z}_{n}^{\mathbb{Z}_{n}}$,
with at least one fixed point
\[
\text{Aut}\left(\prod_{0\le i\ne j<n}\left(\big(x_{f\left(j\right)}-x_{j}\big)^{2}-\big(x_{f\left(i\right)}-x_{i}\big)^{2}\right)\right)=\begin{cases}
\begin{array}{cc}
\text{S}_{n}\text{ if }\left|\left\{ \big(x_{f\left(i\right)}-x_{i}\big)^{2}:i\in\mathbb{Z}_{n}\right\} \right|<n\\
\\
\text{Aut}\left(G_{f}\cup G_{f^{\top}}\backslash\left\{ \left(u,f(u)\right):u=f(u)\right\} \right) & \text{otherwise}
\end{array},\end{cases}
\]
where $G_{f}\cup G_{f^{\top}}\backslash\left\{ \left(u,f\left(u\right)\right):u=f\left(u\right)\right\} $
denotes the loopless bi--directed digraph which underlies the functional
directed graph $G_{f}$.
\end{prop}

\begin{proof}
For notational convenience let 
\[
p_{f}\left(x_{0},\cdots,x_{n-1}\right)=\prod_{0\le i<j<n}\left(\left(x_{f\left(j\right)}-x_{j}\right)^{2}-\left(x_{f\left(i\right)}-x_{i}\right)^{2}\right)
\]
If the underlying undirected graph of $G_{f}\cup G_{f^{\top}}$ has
fewer than $n$ edges, then $\left|\left\{ \left(x_{f\left(i\right)}-x_{i}\right)^{2}:i\in\mathbb{Z}_{n}\right\} \right|<n$
and $p_{f}$ vanishes identically. Thus Aut$\left(p_{f}\right)=$
S$_{n}$. So assume $\left|\left\{ \left(x_{f\left(i\right)}-x_{i}\right)^{2}:i\in\mathbb{Z}_{n}\right\} \right|=n$
and thus $p_{f}$ does not vanish identically. Observe that for all
$\gamma\in$ S$_{n}$, we have
\[
\prod_{0\le i\ne j<n}\left(\left(x_{f\left(j\right)}-x_{j}\right)^{2}-\left(x_{f\left(i\right)}-x_{i}\right)^{2}\right)=\prod_{0\le\gamma^{-1}\left(i\right)\ne\gamma^{-1}\left(j\right)<n}\left(\left(x_{f\gamma^{-1}\left(j\right)}-x_{\gamma^{-1}\left(j\right)}\right)^{2}-\left(x_{f\gamma^{-1}\left(i\right)}-x_{\gamma^{-1}\left(i\right)}\right)^{2}\right).
\]
The right action does not change the polynomial because it simply
rearranges the factors as follows\textbf{
\[
\prod_{0\le i\ne j<n}\left(\left(x_{f\left(j\right)}-x_{j}\right)^{2}-\left(x_{f\left(i\right)}-x_{i}\right)^{2}\right)=\prod_{0\le i\ne j<n}\left(\left(x_{f\gamma^{-1}\left(j\right)}-x_{\gamma^{-1}\left(j\right)}\right)^{2}-\left(x_{f\gamma^{-1}\left(i\right)}-x_{\gamma^{-1}\left(i\right)}\right)^{2}\right).
\]
}Acting on $p_{f}$ by permuting it's variables effects a left action
so that
\[
\prod_{0\le i\ne j<n}\left(\left(x_{\gamma f\left(j\right)}-x_{\gamma\left(j\right)}\right)^{2}-\left(x_{\gamma f\left(i\right)}-x_{\gamma\left(i\right)}\right)^{2}\right)=\prod_{0\le i\ne j<n}\left(\left(x_{\gamma f\gamma^{\left(-1\right)}\left(j\right)}-x_{\gamma\gamma^{\left(-1\right)}\left(j\right)}\right)^{2}-\left(x_{\gamma f\gamma^{\left(-1\right)}\left(i\right)}-x_{\gamma\gamma^{\left(-1\right)}\left(i\right)}\right)^{2}\right),
\]
\[
\implies\prod_{0\le i\ne j<n}\left(\left(x_{\gamma f\left(j\right)}-x_{\gamma\left(j\right)}\right)^{2}-\left(x_{\gamma f\left(i\right)}-x_{\gamma\left(i\right)}\right)^{2}\right)=\prod_{0\le i\ne j<n}\left(\left(x_{\gamma f\gamma^{\left(-1\right)}\left(j\right)}-x_{j}\right)^{2}-\left(x_{\gamma f\gamma^{\left(-1\right)}\left(i\right)}-x_{i}\right)^{2}\right).
\]
Thus $\left(p_{f}(x_{\gamma(0)},\dots,x_{\gamma(n-1)})\right)^{2}=\left(p_{\gamma f\gamma^{-1}}(x_{0},\dots,x_{n-1})\right)^{2}$.
If on the one hand $\sigma\in$ Aut$\left(G_{f}\cup G_{f^{\top}}\backslash\left\{ \left(u,f\left(u\right)\right):u=f\left(u\right)\right\} \right)$,
then 
\[
\left(p_{f}\right)^{2}=\left(p_{\sigma f\sigma^{-1}}\right)^{2}=\left(p_{f}(x_{\sigma(0)},\dots,x_{\sigma(n-1)})\right)^{2}
\]
and $\sigma\in$ Aut$\left((p_{f})^{2}\right)$. If on the other hand
$\sigma\notin$ Aut$\left(G_{f}\cup G_{f^{\top}}\backslash\left\{ \left(u,f\left(u\right)\right):u=f\left(u\right)\right\} \right)$,
we have 
\[
\left(p_{f}\right)^{2}\ne\left(p_{\sigma f\sigma^{-1}}\right)^{2}=\left(p_{f}(x_{\sigma(0)},\dots,x_{\sigma(n-1)})\right)^{2}
\]
 and $\sigma\notin$ Aut$\left((p_{f})^{2}\right)$. Since $p_{f}$
determines\footnote{We discuss this further in the appendix} the
edges set of $G_{f}\cup G_{f^{\top}}\backslash\left\{ \left(u,f\left(u\right)\right):u=f\left(u\right)\right\} $.
\end{proof}
\begin{lem}[Zeilberger Birthday Lemma\footnote{The need for this lemma was pointed out to the author by Doron Zeilberger
on July 2nd 2022 which happens to coincide with his birthday.}]
\label{ZBL} Let $f\in\mathbb{Z}_{n}^{\mathbb{Z}_{n}}$ be subject
to the fixed point condition $f^{\left(n-1\right)}\left(\mathbb{Z}_{n}\right)=\left\{ r\right\} $
where $n\ge3$ and let
\[
F_{f}\left(\mathbf{x}\right)=\prod_{0\le i<j<n}\left(x_{j}-x_{i}\right)\left(\left(x_{f\left(j\right)}-x_{j}\right)^{2}-\left(x_{f\left(i\right)}-x_{i}\right)^{2}\right).
\]
If 
\[
0\not\equiv F_{f}\left(\mathbf{x}\right)\mod\left\{ \left(x_{u}\right)^{\underline{n}}:u\in\mathbb{Z}_{n}\right\} ,
\]
then
\[
\text{Aut}\left(\sum_{g\in\mathbb{Z}_{n}^{\mathbb{Z}_{n}}}\left(F_{f}(g)\right)^{2}L_{g}\left(\mathbf{x}\right)\right)\supseteq\text{Aut}\left(G_{f}\cup G_{f^{\top}}\backslash\left(r,r\right)\right).
\]
\end{lem}

\begin{proof}
We start by showing that for all permutations $\gamma\in$ S$_{n}$
we have 
\[
\left(F_{f}(x_{\gamma\left(0\right)},\cdots,x_{\gamma\left(n-1\right)})\right)^{2}=\left(F_{\gamma f\gamma^{-1}}(x_{0},\cdots,x_{n-1})\right)^{2}.
\]
Observe that, by commutativity of the product, the right action by
any $\gamma\in$ S$_{n}$ leaves $\left(F_{f}(x_{0},\cdots,x_{n-1})\right)^{2}$
unchanged since
\[
\prod_{0\le i\ne j<n}\left(\left(x_{f\left(j\right)}-x_{j}\right)^{2}-\left(x_{f\left(i\right)}-x_{i}\right)^{2}\right)^{2}=\prod_{0\le\gamma^{-1}\left(i\right)\ne\gamma^{-1}\left(j\right)<n}\left(\left(x_{f\gamma^{-1}\left(j\right)}-x_{\gamma^{-1}\left(j\right)}\right)^{2}-\left(x_{f\gamma^{-1}\left(i\right)}-x_{\gamma^{-1}\left(i\right)}\right)^{2}\right).
\]
Thus for all $\gamma\in$ S$_{n}$ by commutativity of products we
have\textbf{
\[
\prod_{0\le i<j<n}\left(\left(x_{f\left(j\right)}-x_{j}\right)^{2}-\left(x_{f\left(i\right)}-x_{i}\right)^{2}\right)^{2}=\prod_{0\le i<j<n}\left(\left(x_{f\gamma^{-1}\left(j\right)}-x_{\gamma^{-1}\left(j\right)}\right)^{2}-\left(x_{f\gamma^{-1}\left(i\right)}-x_{\gamma^{-1}\left(i\right)}\right)^{2}\right)^{2}.
\]
}Consequently, left action by any $\gamma\in$ S$_{n}$ effects an
action by conjugation on $f$ as follows
\[
\prod_{0\le i<j<n}\left(\left(x_{\gamma f\left(j\right)}-x_{\gamma\left(j\right)}\right)^{2}-\left(x_{\gamma f\left(i\right)}-x_{\gamma\left(i\right)}\right)^{2}\right)^{2}=\prod_{0\le i<j<n}\left(\left(x_{\gamma f\gamma^{\left(-1\right)}\left(j\right)}-x_{\gamma\gamma^{\left(-1\right)}\left(j\right)}\right)^{2}-\left(x_{\gamma f\gamma^{\left(-1\right)}\left(i\right)}-x_{\gamma\gamma^{\left(-1\right)}\left(i\right)}\right)^{2}\right)^{2},
\]
Thus, remainders of $\left(F_{f}(x_{\gamma\left(0\right)},\cdots,x_{\gamma\left(n-1\right)})\right)^{2}$
and $\left(F_{\gamma f\gamma^{-1}}(x_{0},\cdots,x_{n-1})\right)^{2}$
are equal. By the corollary \ref{Quotient_Remainder_Corollary}, the
polynomial $\left(F_{f}\right)^{2}$ admits an expansion of the form
\[
\left(F_{f}(\mathbf{x})\right)^{2}=\sum_{0<k\le d+1-n}\left(\sum_{g_{k}\in\mathbb{Z}_{n+k}^{\mathbb{Z}_{n}}}b_{g_{k}}L_{g_{k}}(\mathbf{x})\right)+\sum_{\begin{array}{c}
\sigma\in\text{S}_{n}\\
G_{\sigma f\sigma^{-1}}\in\text{GrL}\left(G_{f}\right)
\end{array}}\left(F_{f}(\sigma)\right)^{2}L_{\sigma}\left(\mathbf{x}\right),
\]
where $n\le d\le6{n \choose 2}$. Thus the congruence identity relating
$\left(F_{f}(\mathbf{x})\right)^{2}$ to its canonical representative
is
\[
\left(F_{f}(\mathbf{x})\right)^{2}\equiv\sum_{\begin{array}{c}
\sigma\in\text{S}_{n}\\
G_{\sigma f\sigma^{-1}}\in\text{GrL}\left(G_{f}\right)
\end{array}}\left(F_{f}(\sigma)\right)^{2}L_{\sigma}\left(\mathbf{x}\right)\mod\left\{ \left(x_{i}\right)^{\underline{n}}:i\in\mathbb{Z}_{n}\right\} .
\]
The right hand side of the following congruence identities are expressions
of the canonical representative of the corresponding congruence class
\[
\begin{array}{ccc}
\left(F_{f}(x_{\gamma\left(0\right)},\cdots,x_{\gamma\left(n-1\right)})\right)^{2} & \equiv & \underset{\begin{array}{c}
\sigma\in\text{S}_{n}\\
G_{\sigma f\sigma^{-1}}\in\text{GrL}\left(G_{f}\right)
\end{array}}{\sum}\left(F_{f}(\sigma)\right)^{2}\underset{\begin{array}{c}
u\in\mathbb{Z}_{n}\\
v_{u}\in\mathbb{Z}_{n}\backslash\left\{ \sigma\left(u\right)\right\} 
\end{array}}{\prod}\left(\frac{x_{\gamma\left(u\right)}-v_{u}}{\sigma\left(u\right)-v_{u}}\right),\\
\\
\\
 & \equiv & \underset{\begin{array}{c}
\sigma\in\text{S}_{n}\\
G_{\sigma f\sigma^{-1}}\in\text{GrL}\left(G_{f}\right)
\end{array}}{\sum}\left(F_{f}(\sigma)\right)^{2}L_{\sigma\gamma^{-1}}\left(\mathbf{x}\right),\\
\\
\\
 & \equiv & \underset{\begin{array}{c}
\sigma\in\text{S}_{n}\\
G_{\sigma\gamma^{-1}\gamma f\gamma^{-1}\gamma\sigma^{-1}}\in\text{GrL}\left(G_{f}\right)
\end{array}}{\sum}\left(F_{f}(\sigma\gamma^{-1}\gamma)\right)^{2}L_{\sigma\gamma^{-1}}\left(\mathbf{x}\right),\\
\\
\\
 & \equiv & \underset{\begin{array}{c}
\theta\in\text{S}_{n}\\
G_{\theta\gamma f\gamma^{-1}\theta^{-1}}\in\text{GrL}\left(G_{f}\right)
\end{array}}{\sum}\left(F_{f}(\theta\gamma)\right)^{2}L_{\theta}\left(\mathbf{x}\right),\\
\\
\\
 & \equiv & \underset{\begin{array}{c}
\theta\in\text{S}_{n}\\
G_{\theta\gamma f\gamma^{-1}\theta^{-1}}\in\text{GrL}\left(G_{f}\right)
\end{array}}{\sum}\left(F_{\gamma f\gamma^{-1}}(\theta)\right)^{2}L_{\theta}\left(\mathbf{x}\right).
\end{array}
\]
\[
\implies\text{Aut}\left(\sum_{g\in\mathbb{Z}_{n}^{\mathbb{Z}_{n}}}\left(F_{f}(g)\right)^{2}L_{g}\left(\mathbf{x}\right)\right)\supseteq\text{Aut}\left((F_{f})^{2}\right)
\]
By proposition (\ref{Stabiliser_Subgroup}), we have
\[
\text{Aut}\left((F_{f})^{2}\right)=\text{Aut}\left(G_{f}\cup G_{f^{\top}}\backslash\left(r,r\right)\right).
\]
We conclude that 
\[
\text{Aut}\left(\sum_{g\in\mathbb{Z}_{n}^{\mathbb{Z}_{n}}}\left(F_{f}(g)\right)^{2}L_{g}\left(\mathbf{x}\right)\right)\supseteq\text{Aut}\left(G_{f}\cup G_{f^{\top}}\backslash\left(r,r\right)\right),
\]
as claimed.
\end{proof}
As a corollary, given
\[
f\in\left\{ h\in\mathbb{Z}_{n}^{\mathbb{Z}_{n}}:\begin{array}{c}
h\left(0\right)=0\\
h\left(i\right)<i,\,\forall\,i\in\mathbb{Z}_{n}\backslash\left\{ 0\right\} 
\end{array}\right\} ,
\]
subject to $\left|f^{-1}\left(\left\{ f\left(n-1\right)\right\} \right)\right|>1$
and
\[
f^{-1}\left(\left\{ f\left(n-1\right)\right\} \right)=\left\{ n-1,...,n-\left|f^{-1}\left(\left\{ f\left(n-1\right)\right\} \right)\right|\right\} ,
\]
 If $F_{f}\not\equiv0\mod\left\{ \left(x_{u}\right)^{\underline{n}}:u\in\mathbb{Z}_{n}\right\} $
then the transposition which exchanges the variables $x_{u}$ with
$x_{v}$ where $u,v\in f^{-1}\left(\left\{ f\left(n-1\right)\right\} \right)$
necessarily lies in
\[
\text{Aut}\left(\sum_{g\in\mathbb{Z}_{n}^{\mathbb{Z}_{n}}}F_{f}(g)\,L_{g}\left(x_{0},\cdots,x_{n-1}\right)\right).
\]
Furthermore, the support of non-vanishing Lagrange basis polynomials
partitions into disjoint cosets of $\text{Aut}\left(G_{f}\cup G_{f^{\top}}\backslash\left(0,0\right)\right)$
as follows
\[
F_{f}\left(\mathbf{x}\right)\equiv\left(\sum_{g\in\mathbb{Z}_{n}^{\mathbb{Z}_{n}}}F_{f}\left(g\right)L_{g}\left(\mathbf{x}\right)\right)=
\]
\[
\prod_{v\in\mathbb{Z}_{n}}\left(\left(v!\right)^{2}\frac{\left(n-1+v\right)!}{\left(2v\right)!}\right)\sum_{\begin{array}{c}
\sigma\in\nicefrac{\text{S}_{n}}{\text{Aut}\left(G_{f}\cup G_{f^{\top}}\backslash\left(0,0\right)\right)}\\
G_{\sigma f\sigma^{-1}}\in\text{GrL}\left(G_{f}\right)
\end{array}}\sum_{\theta_{\sigma}\in\text{Aut}\left(G_{f}\cup G_{f^{\top}}\backslash\left(0,0\right)\right)}\text{sgn}\left(\left|\left(\sigma\theta_{\sigma}^{-1}\right)f\left(\sigma\theta_{\sigma}^{-1}\right)^{-1}-\text{id}\right|\circ\sigma\theta_{\sigma}^{-1}\right)L_{\sigma\theta_{\sigma}^{-1}}\left(\mathbf{x}\right).
\]
By the complementary labeling symmetry, if
\[
\left(n-1-\text{id}\right)\circ\text{Aut}\left(G_{f}\cup G_{f^{\top}}\backslash\left(r,r\right)\right)\circ\left(n-1-\text{id}\right)^{-1}=\text{Aut}\left(G_{f}\cup G_{f^{\top}}\backslash\left(r,r\right)\right),
\]
then 
\[
\text{Aut}\left(\sum_{g\in\mathbb{Z}_{n}^{\mathbb{Z}_{n}}}\left(F_{f}(g)\right)^{2}L_{g}\left(\mathbf{x}\right)\right)\supseteq\text{Aut}\left(G_{f}\cup G_{f^{\top}}\backslash\left(r,r\right)\right)\cup\text{Aut}\left(G_{f}\cup G_{f^{\top}}\backslash\left(r,r\right)\right)\circ\left(n-1-\text{id}\right).
\]

\begin{example}
As an illustration take the function $f\in\mathbb{Z}_{4}^{\mathbb{Z}_{4}}$
such that
\[
f\left(0\right)=0,\,f\left(1\right)=0,\,f\left(2\right)=1,\,f\left(3\right)=2,
\]
\[
\left(F_{f}(x_{0},x_{1},x_{2},x_{3})\right)^{2}=\prod_{0\le i<j<4}\left(x_{j}-x_{i}\right)^{2}\left(\left(x_{f\left(j\right)}-x_{j}\right)^{2}-\left(x_{f\left(i\right)}-x_{i}\right)^{2}\right)^{2}.
\]
We see that 
\[
\text{Aut}\left(F_{f}(x_{0},x_{1},x_{2},x_{3})\right)^{2}=\left\{ \text{id},\,\sigma_{1}\right\} 
\]
where the directed edge set of the functional digraph $G_{\sigma_{1}}$
is
\[
E\left(G_{\sigma_{1}}\right)=\left\{ \left(0,3\right),\left(1,2\right),\left(2,1\right),\left(3,0\right)\right\} .
\]
We see that
\[
\left(\sum_{g\in\mathbb{Z}_{4}^{\mathbb{Z}_{4}}}\left(F_{f}(g)\right)^{2}L_{g}\left(\mathbf{x}\right)\right)=
\]
\[
74649600\,x_{0}^{3}x_{1}^{3}x_{2}^{3}x_{3}^{3}-335923200\,x_{0}^{3}x_{1}^{3}x_{2}^{3}x_{3}^{2}-335923200\,x_{0}^{3}x_{1}^{3}x_{2}^{2}x_{3}^{3}+\cdots-4031078400\,x_{0}x_{2}x_{3}-4031078400\,x_{1}x_{2}x_{3},
\]
and
\[
\text{Aut}\left(\sum_{g\in\mathbb{Z}_{4}^{\mathbb{Z}_{4}}}\left(F_{f}(g)\right)^{2}L_{g}\left(\mathbf{x}\right)\right)=\left\{ \text{id},\,\sigma_{1},\,\sigma_{2},\sigma_{3}\right\} ,
\]
where the directed edge set of functional digraphs $G_{\sigma_{2}}$
and $G_{\sigma_{3}}$ are respectively
\[
E\left(G_{\sigma_{2}}\right)=\left\{ \left(0,1\right),\left(1,0\right),\left(2,3\right),\left(3,2\right)\right\} ,\:E\left(G_{\sigma_{3}}\right)=\left\{ \left(0,2\right),\left(1,3\right),\left(2,0\right),\left(3,1\right)\right\} 
\]

The example above illustrates a setting where the automorphism group
of $(F_{f})^{2}$ i.e. $\text{Aut}\left((F_{f})^{2}\right)$ is a
proper subgroup of the automorphism group of its canonical representative
i.e.
\[
\text{Aut}\left(\sum_{g\in\mathbb{Z}_{4}^{\mathbb{Z}_{4}}}\left(F_{f}(g)\right)^{2}L_{g}\left(\mathbf{x}\right)\right).
\]
\begin{figure}
\begin{tikzpicture}
%% Vertices of the graph
\node (3) at (-2,0) {};
\node (2) at (0 ,0) {};
\node (1) at (2 ,0) {};
\node (0) at (4 ,0) {};
\draw[fill=black] (-2,0) circle (3pt);
\draw[fill=black] (0 ,0) circle (3pt);
\draw[fill=black] (2 ,0) circle (3pt);
\draw[fill=black] (4 ,0) circle (3pt);

%% Vertex labels of the graph
\node at (-2,0.5) {$3$};
\node at (0 ,0.5) {$2$};
\node at (2 ,0.5) {$1$};
\node at (4 ,0.5) {$0$};

%% Functional directed graph name
\node at (-3,0) {$G_f =$};

%%% directed edges
\draw (3) edge[thick,->] (2);
\draw (2) edge[thick,->] (1) edge[thick,->] (0);
\draw (0) edge[thick,->,out=45,in=315,looseness=10] (0);
\end{tikzpicture} \centering \caption{$f\left(0\right)=0,\,f\left(1\right)=0,\,f\left(2\right)=1,\,f\left(3\right)=2$}
\end{figure}
Given $P\in\mathbb{Q}\left[x_{0},\cdots,x_{n-1}\right]$, recall that
$P$ depends on the variable $x_{u}$ if the polynomial $\frac{\partial P}{\partial x_{u}}$
does not vanish identically. In other words the expanded form of $P$
features at least one monomial multiple of the variable $x_{u}$ whose
coefficient does not vanish.
\end{example}

\begin{prop}
\label{prop:dependencies}Let $P\left(x_{0},\cdots,x_{n-1}\right)\in\mathbb{Q}\left[x_{0},\cdots,x_{n-1}\right]$
be dependent only on the subset of variables in
\[
\left\{ x_{i}:i\in S\subsetneq\mathbb{Z}_{n}\right\} ,
\]
If $P\left(\mathbf{x}\right)$ is of degree at most $n-1$ in the
said variables, then for any positive integer $m$ the canonical representative
of 
\[
\left(P(\mathbf{x})\right)^{m}\mod\left\{ \left(x_{i}\right)^{\underline{n}}:i\in S\right\} 
\]
can depend only on variables in the subset $\left\{ x_{i}:i\in S\right\} $.
\end{prop}

\begin{proof}
By our premise $P$ equals it own canonical representative i.e.
\[
P\left(\mathbf{x}\right)=\left(\sum_{g\in\mathbb{Z}_{n}^{\mathbb{Z}_{n}}}P\left(g\right)\,L_{g}\left(\mathbf{x}\right)\right)=\sum_{g\in\mathbb{Z}_{n}^{S}}P\left(g\right)\,\prod_{i\in S}\left(\prod_{j_{i}\in\mathbb{Z}_{n}\backslash\left\{ g\left(i\right)\right\} }\bigg(\frac{x_{i}-j_{i}}{g\left(i\right)-j_{i}}\bigg)\right).
\]
\[
\implies\left(P(\mathbf{x})\right)^{m}=\left(\sum_{g\in\mathbb{Z}_{n}^{S}}P\left(g\right)\,\prod_{i\in S}\left(\prod_{j_{i}\in\mathbb{Z}_{n}\backslash\left\{ g\left(i\right)\right\} }\bigg(\frac{x_{i}-j_{i}}{g\left(i\right)-j_{i}}\bigg)\right)\right)^{m}.
\]
We see that $\left(P(\mathbf{x})\right)^{m}$ depends only on variables
in $\left\{ x_{i}:i\in S\right\} $. The canonical representative
of $\left(P(\mathbf{x})\right)^{m}$ is 
\[
\sum_{g\in\mathbb{Z}_{n}^{\mathbb{Z}_{n}}}\left(P(g)\right)^{m}\,L_{g}\left(\mathbf{x}\right)
\]
 By the proposition \ref{Quotient_Remainder_Expansion}, the canonical
representative of $\left(P(\mathbf{x})\right)^{m}$ can be obtained
by reducing $\left(P(\mathbf{x})\right)^{m}$ modulo algebraic relations
\[
\left\{ \left(x_{i}\right)^{\underline{n}}:i\in S\right\} .
\]
Accordingly, the canonical representative of $\left(P(\mathbf{x})\right)^{m}$
is devised by repeatedly replacing into the expanded form of $\left(P(\mathbf{x})\right)^{m}$
every occurrence of $\left(x_{i}\right)^{n}$ with $\left(x_{i}\right)^{n}-\left(x_{i}\right)^{\underline{n}}$
for all $i\in S$ until we obtain a polynomial of degree $<n$ in
each variable. The reduction procedure never introduces a variable
in the complement of the set $\left\{ x_{i}:i\in S\right\} $. Therefore,
the canonical representative of $\left(P(\mathbf{x})\right)^{m}$
given by
\[
\left(\sum_{g\in\mathbb{Z}_{n}^{\mathbb{Z}_{n}}}\left(P(g)\right)^{m}\,L_{g}\left(\mathbf{x}\right)\right)=\sum_{g\in\mathbb{Z}_{n}^{S}}\big(P(g)\big)^{m}\prod_{i\in S}\left(\prod_{j_{i}\in\mathbb{Z}_{n}\backslash\left\{ g\left(i\right)\right\} }\bigg(\frac{x_{i}-j_{i}}{g\left(i\right)-j_{i}}\bigg)\right).
\]
depends only on variables in $\left\{ x_{i}:i\in S\right\} $ as claimed.
\end{proof}
\begin{lem}[The monomial overlapping lemma]
\label{MnOvLe} Let $\mathbf{x}$ denote the sequence of variables
$(x_{0},\ldots,x_{n-1})$. For any non-empty $\mathcal{S}\subseteq\mathrm{S}_{n}$
such that $a_{\sigma}\in\mathbb{Q}\setminus\{0\}$ for all $\sigma\in\mathcal{S}$,
we have
\[
\sum_{\sigma\in\mathcal{S}}a_{\sigma}\,L_{\sigma}\left(\mathbf{x}\right)=\sum_{f\in\mathcal{M}_{\mathcal{S}}}c_{f}\prod_{i\in\mathbb{Z}_{n}}x_{i}^{f\left(i\right)},
\]
where $c_{f}\in\mathbb{Q}\setminus\{0\}$ and $\left|f^{-1}\left(\left\{ 0\right\} \right)\right|\le1$
for all $f\in\mathcal{M}_{\mathcal{S}}$. The right-hand side above
expresses the left-hand side as a linear combination of monomials.
\end{lem}

\begin{proof}
Stated otherwise, the present lemma asserts that every term in the
expanded form is a multiple of at least $\left(n-1\right)$ distinct
variables. Consider the Lagrange basis polynomial associated with
any $\sigma\in\mathcal{S}$:
\[
L_{\sigma}({\bf x})=\prod_{\substack{i\in\mathbb{Z}_{n}\\
j_{i}\in\mathbb{Z}_{n}\setminus\left\{ \sigma\left(i\right)\right\} 
}
}\left(\frac{x_{i}-j_{i}}{\sigma\left(i\right)-j_{i}}\right)=\prod_{i\in\mathbb{Z}_{n}\setminus\left\{ \sigma^{-1}\left(0\right)\right\} }\left(\prod_{j_{i}\in\mathbb{Z}_{n}\setminus\left\{ \sigma\left(i\right)\right\} }\frac{x_{i}-j_{i}}{\sigma\left(i\right)-j_{i}}\right)\,\prod_{j_{\sigma^{-1}\left(0\right)}\in\mathbb{Z}_{n}\backslash\left\{ 0\right\} }\left(\frac{x_{\sigma^{-1}\left(0\right)}-j}{0-j}\right).
\]
On the right-hand side of the second equal sign immediately above,
the univariate polynomial encompassed within the scope of third $\Pi$
indexed by $j_{\sigma^{-1}\left(0\right)}\in\mathbb{Z}_{n}\backslash\left\{ 0\right\} $
has (in its expanded form) a non-vanishing constant term equal to
one. However, the constant term vanishes within the expanded form
of each univariate factors $\underset{j_{i}\in\mathbb{Z}_{n}\setminus\left\{ \sigma\left(i\right)\right\} }{\prod}\frac{x_{i}-j_{i}}{\sigma\left(i\right)-j_{i}}$
encompassed within the scope of the first $\Pi$ indexed by $i\in\mathbb{Z}_{n}\setminus\left\{ \sigma^{-1}\left(0\right)\right\} $:
\[
L_{\sigma}({\bf x})=\underbrace{\prod_{\substack{i\in\mathbb{Z}_{n}\setminus\left\{ \sigma^{-1}\left(0\right)\right\} \\
j_{i}\in\mathbb{Z}_{n}\setminus\left\{ \sigma\left(i\right)\right\} 
}
}\bigg(\frac{x_{i}-j_{i}}{\sigma\left(i\right)-j_{i}}\bigg)}_{\text{does not feature the variable }x_{\sigma^{-1}\left(0\right)}}\:\bigg(\frac{(x_{\sigma^{-1}\left(0\right)})^{n-1}+\cdots+(-1)^{n-1}(n-1)!}{(-1)^{n-1}(n-1)!}\bigg).
\]
Observe that each summand in the expanded form of the Lagrange basis
polynomial $L_{\sigma}({\bf x})$ above which is a non--vanishing
monomial multiple of $x_{\sigma^{-1}\left(0\right)}$ is a multiple
of every variable in $\left\{ x_{0},\ldots,x_{n-1}\right\} $ whereas
by contrast every non--vanishing monomial which is not a multiple
of $x_{\sigma^{-1}\left(0\right)}$ is a multiple of every other variables
i.e. variables in the set $\left\{ x_{0},\ldots,x_{n-1}\right\} \backslash\left\{ x_{\sigma^{-1}\left(0\right)}\right\} $.
Applying the same argument to each $\sigma\in\mathcal{S}$ yields
the desired claim.
\end{proof}
We now state and prove the composition lemma.
\begin{lem}[The Composition Lemma]
\label{CL} Let $n$ be a positive integer greater than $3$. For
all functions $f\in\mathbb{Z}_{n}^{\mathbb{Z}_{n}}$, If $\left|f^{\left(n-1\right)}\left(\mathbb{Z}_{n}\right)\right|=1$
and $G_{f}$ has diameter $\ge3$, then
\[
\max_{\sigma\in\text{S}_{n}}\left|\left\{ \left|\sigma f^{\left(2\right)}\sigma^{(-1)}\left(i\right)-i\right|:i\in\mathbb{Z}_{n}\right\} \right|\le\max_{\sigma\in\text{S}_{n}}\left|\left\{ \left|\sigma f\sigma^{(-1)}\left(i\right)-i\right|:i\in\mathbb{Z}_{n}\right\} \right|.
\]
\end{lem}

\begin{proof}
Assume without loss of generality $f$ lies in the semigroup
\[
\left\{ h\in\mathbb{Z}_{n}^{\mathbb{Z}_{n}}:\begin{array}{c}
h\left(0\right)=0\\
h\left(i\right)<i,\,\forall\,i\in\mathbb{Z}_{n}\backslash\left\{ 0\right\} 
\end{array}\right\} .
\]
For we see that if $f\in\mathbb{Z}_{n}^{\mathbb{Z}_{n}}$ subject
to $\left|f^{\left(n-1\right)}\left(\mathbb{Z}_{n}\right)\right|=1$,
does not lie in the semigroup, there exist a permutation $\sigma\in\text{S}_{n}$
of vertex labels in $G_{f}$ by which we devise $G_{\sigma f\sigma^{-1}}$
isomorphic to $G_{f}$ such that
\[
\sigma f\sigma^{-1}\in\left\{ h\in\mathbb{Z}_{n}^{\mathbb{Z}_{n}}:\begin{array}{c}
h\left(0\right)=0\\
h\left(i\right)<i,\,\forall\,i\in\mathbb{Z}_{n}\backslash\left\{ 0\right\} 
\end{array}\right\} .
\]
Thus functional directed graphs of members of the semigroup
\[
\left\{ h\in\mathbb{Z}_{n}^{\mathbb{Z}_{n}}:\begin{array}{c}
h\left(0\right)=0\\
h\left(i\right)<i,\,\forall\,i\in\mathbb{Z}_{n}\backslash\left\{ 0\right\} 
\end{array}\right\} ,
\]
account for at least one member of every conjugacy class of functional
trees. Observe that for any $f$ in the semigroup the iterate $f^{\left(2^{\left\lceil \log_{2}\left(n-1\right)\right\rceil }\right)}$
is the identically zero function. Hence, the existence of some function
\[
g\in\left\{ h\in\mathbb{Z}_{n}^{\mathbb{Z}_{n}}:\begin{array}{c}
h\left(0\right)=0\\
h\left(i\right)<i,\,\forall\,i\in\mathbb{Z}_{n}\backslash\left\{ 0\right\} 
\end{array}\right\} \text{ such that }n>\max_{\sigma\in\text{S}_{n}}\left|\left\{ \left|\sigma g\sigma^{(-1)}\left(i\right)-i\right|:i\in\mathbb{Z}_{n}\right\} \right|,
\]
implies the existence of some function
\[
f\in\left\{ g^{\left(2^{\kappa}\right)}\,:\,0\le\kappa<\left\lceil \log_{2}\left(n-1\right)\right\rceil \right\} \subset\left\{ h\in\mathbb{Z}_{n}^{\mathbb{Z}_{n}}:\begin{array}{c}
h\left(0\right)=0\\
h\left(i\right)<i,\,\forall\,i\in\mathbb{Z}_{n}\backslash\left\{ 0\right\} 
\end{array}\right\} ,
\]
such that
\[
\max_{\sigma\in\text{S}_{n}}\left|\left\{ \left|\sigma f^{\left(2\right)}\sigma^{(-1)}\left(i\right)-i\right|:i\in\mathbb{Z}_{n}\right\} \right|>\max_{\sigma\in\text{S}_{n}}\left|\left\{ \left|\sigma f\sigma^{(-1)}\left(i\right)-i\right|:i\in\mathbb{Z}_{n}\right\} \right|.
\]
Consequently if the purported claim of lemma \ref{CL} holds, then
there can be no member $g$ of the semigroup for which 
\[
n>\max_{\sigma\in\text{S}_{n}}\left|\left\{ \left|\sigma g\sigma^{(-1)}\left(i\right)-i\right|:i\in\mathbb{Z}_{n}\right\} \right|.
\]
We now proceed to prove a generalization of the desired claim. Assume
without loss of generality that the vertex labeled $n-1$ is at maximum
edge distance from the vertex labeled $0$ (the root node) in $G_{f}$.
Also assume without loss of generality that vertices in the set $f^{-1}\left(\left\{ f\left(n-1\right)\right\} \right)$
(namely the vertex set made up of $n-1$ and its sibling nodes in
$G_{f}$) are assigned the largest possible labels from $\mathbb{Z}_{n}$
in other words
\[
f^{-1}\left(\left\{ f\left(n-1\right)\right\} \right)=\left\{ n-1,\,n-2,\,...,\,n-\left|f^{-1}\left(\left\{ f\left(n-1\right)\right\} \right)\right|\right\} \;\text{ and}
\]
\[
f\left(n-\left|f^{-1}\left(\left\{ f\left(n-1\right)\right\} \right)\right|\right)=n-\left|f^{-1}\left(\left\{ f\left(n-1\right)\right\} \right)\right|-1.
\]
For if this was not the case, we can effect a permutation $\sigma$
of vertex labels in $G_{f}$ to devise an isomorphic functional directed
graph $G_{\sigma f\sigma^{-1}}$ such that $\sigma f\sigma^{-1}$
lies in the semigroup and both latter conditions are met. Now consider
the local iteration which devises from $G_{f}$ a new functional directed
graph $G_{g}$ obtained by sliding one edge length closer to the root
node the collection of sibling nodes which include the vertex labeled
$n-1$ i.e. vertices in the set $f^{-1}\left(\left\{ f\left(n-1\right)\right\} \right)$.
More precisely, $G_{g}$ is the functional directed graph of $g$
defined such that for all $i\in\mathbb{Z}_{n}$ 
\[
g\left(i\right)=\begin{cases}
\begin{array}{cc}
f^{\left(2\right)}\left(i\right) & \text{ if }i\in f^{-1}\left(\left\{ f\left(n-1\right)\right\} \right)\\
f\left(i\right) & \text{otherwise}
\end{array},\ \forall\,i\in\mathbb{Z}_{n}.\end{cases}
\]
By construction
\[
g\in\left\{ h\in\mathbb{Z}_{n}^{\mathbb{Z}_{n}}:\begin{array}{c}
h\left(0\right)=0\\
h\left(i\right)<i,\,\forall\,i\in\mathbb{Z}_{n}\backslash\left\{ 0\right\} 
\end{array}\right\} .
\]
We show that 
\[
\max_{\sigma\in\text{S}_{n}}\left|\left\{ \left|\sigma g\sigma^{(-1)}\left(i\right)-i\right|:i\in\mathbb{Z}_{n}\right\} \right|\le\max_{\sigma\in\text{S}_{n}}\left|\left\{ \left|\sigma f\sigma^{(-1)}\left(i\right)-i\right|:i\in\mathbb{Z}_{n}\right\} \right|.
\]
The inequality immediately above generalizes the claim of the lemma
\ref{CL} since in the latter inequality, the function $f$ is only
partially iterated. More precisely $f$ is iterated only on the restriction
$f^{-1}\left(\left\{ f\left(n-1\right)\right\} \right)\subset\mathbb{Z}_{n}$.
It is easy to see that the claim of lemma \ref{CL} follows from this
generalization. We prove by contradiction the contrapositive claim
i.e. $\emptyset=\text{GrL}\left(G_{f}\right)\Longrightarrow\text{GrL}\left(G_{g}\right)=\emptyset$.
By construction the polynomial
\[
\begin{array}{cc}
P_{f}\left(\mathbf{x}\right)=\underset{0\le i<j<n}{\prod}\left(x_{j}-x_{i}\right)\times\\
\underset{\begin{array}{c}
0\le u<v\le f\left(n-1\right)\\
t\in\left\{ 0,1\right\} 
\end{array}}{\prod} & \bigg(x_{f\left(v\right)}-x_{v}+\left(-1\right)^{t}(x_{f\left(u\right)}-x_{u})\bigg)\times\\
\underset{\begin{array}{c}
v\in f^{-1}\left(\left\{ f\left(n-1\right)\right\} \right)\\
0\le u\le f\left(n-1\right)\\
t\in\left\{ 0,1\right\} 
\end{array}}{\prod} & \bigg(x_{f\left(v\right)}-x_{v}+\left(-1\right)^{t}(x_{f\left(u\right)}-x_{u})\bigg)\times\\
\underset{\begin{array}{c}
v\in f^{-1}\left(\left\{ f\left(n-1\right)\right\} \right)\\
f\left(n-1\right)<u<v\\
t\in\left\{ 0,1\right\} 
\end{array}}{\prod} & \bigg(x_{f\left(v\right)}-x_{v}+\left(-1\right)^{t}(x_{f\left(u\right)}-x_{u})\bigg),
\end{array}
\]
differs only slightly from
\[
\begin{array}{ccc}
 & P_{g}\left(\mathbf{x}\right)=\underset{0\le i<j<n}{\prod}\left(x_{j}-x_{i}\right)\times\\
 & \underset{\begin{array}{c}
0\le u<v\le f\left(n-1\right)\\
t\in\left\{ 0,1\right\} 
\end{array}}{\prod} & \bigg(x_{f\left(v\right)}-x_{v}+\left(-1\right)^{t}(x_{f\left(u\right)}-x_{u})\bigg)\times\\
 & \underset{\begin{array}{c}
v\in f^{-1}\left(\left\{ f\left(n-1\right)\right\} \right)\\
0\le u\le f\left(n-1\right)\\
t\in\left\{ 0,1\right\} 
\end{array}}{\prod} & \bigg(x_{f^{\left(2\right)}\left(v\right)}-x_{v}+\left(-1\right)^{t}(x_{f\left(u\right)}-x_{u})\bigg)\times\\
 & \underset{\begin{array}{c}
v\in f^{-1}\left(\left\{ f\left(n-1\right)\right\} \right)\\
f\left(n-1\right)<u<v\\
t\in\left\{ 0,1\right\} 
\end{array}}{\prod} & \bigg(x_{f^{\left(2\right)}\left(v\right)}-x_{v}+\left(-1\right)^{t}(x_{f^{\left(2\right)}\left(u\right)}-x_{u})\bigg).
\end{array}
\]
We setup a variable \textbf{telescoping} within each binomial $x_{f^{\left(2\right)}\left(v\right)}-x_{v}$
for all $v\in f^{-1}\left(\left\{ f\left(n-1\right)\right\} \right)$
(associated with an iterated edge) as follows
\[
\begin{array}{c}
\underbrace{{\color{blue}(}x_{f^{\left(2\right)}\left(v\right)}-x_{v}{\color{red})}}\\
x_{v}\longrightarrow x_{f^{\left(2\right)}\left(v\right)}
\end{array}=\begin{array}{c}
\underbrace{{\color{red}(x_{f\left(v\right)}}-x_{v}{\color{red})}}\\
x_{v}\longrightarrow{\color{red}x_{f\left(v\right)}}
\end{array}+\begin{array}{c}
\underbrace{{\color{blue}(}x_{f^{\left(2\right)}\left(v\right)}{\color{blue}-x_{f\left(v\right)}}{\color{blue})}}\\
{\color{blue}x_{f\left(v\right)}}\longrightarrow x_{f^{\left(2\right)}\left(v\right)}
\end{array},
\]
\[
{\color{blue}(}x_{f^{\left(2\right)}\left(v\right)}-x_{v}{\color{red})}={\color{blue}(}x_{f^{\left(2\right)}\left(v\right)}{\color{blue}-x_{f\left(v\right)}}{\color{blue})}+{\color{red}(x_{f\left(v\right)}}-x_{v}{\color{red})}={\color{blue}(}x_{f^{\left(2\right)}\left(n-1\right)}{\color{blue}-x_{f\left(n-1\right)}}{\color{blue})}+{\color{red}(x_{f\left(v\right)}}-x_{v}{\color{red})},
\]
the last equality immediately above results from the fact that $f\left(v\right)=f\left(n-1\right)$
for all $v\in f^{-1}\left(\left\{ f(n-1)\right\} \right)$. Thus
\[
\begin{array}{c}
P_{g}=\underset{0\le i<j<n}{\prod}\left(x_{j}-x_{i}\right)\times\\
\underset{\begin{array}{c}
0\le u<v\le f\left(n-1\right)\\
t\in\left\{ 0,1\right\} 
\end{array}}{\prod}\bigg(x_{f\left(v\right)}-x_{v}+\left(-1\right)^{t}(x_{f\left(u\right)}-x_{u})\bigg)\times\\
\underset{\begin{array}{c}
v\in f^{-1}\left(\left\{ f\left(n-1\right)\right\} \right)\\
0\le u\le f\left(n-1\right)\\
t\in\left\{ 0,1\right\} 
\end{array}}{\prod}\bigg({\color{blue}(}x_{f^{\left(2\right)}\left(n-1\right)}{\color{blue}-x_{f\left(n-1\right)}}{\color{blue})}+{\color{red}(x_{f\left(v\right)}}-x_{v}{\color{red})}+\left(-1\right)^{t}(x_{f\left(u\right)}-x_{u})\bigg)\times\\
\underset{\begin{array}{c}
v\in f^{-1}\left(\left\{ f\left(n-1\right)\right\} \right)\\
f\left(n-1\right)<u<v\\
t\in\left\{ 0,1\right\} 
\end{array}}{\prod}\left({\color{blue}(}x_{f^{\left(2\right)}\left(n-1\right)}{\color{blue}-x_{f\left(n-1\right)})}+{\color{red}(x_{f\left(v\right)}}-x_{v}{\color{red})}+\left(-1\right)^{t}\big({\color{blue}(}x_{f^{\left(2\right)}\left(n-1\right)}{\color{blue}-x_{f\left(n-1\right)})}+{\color{red}(x_{f\left(u\right)}}-x_{u}{\color{red})}\big)\right).
\end{array}
\]
For notational convenience we write 
\begin{equation}
\begin{array}{ccc}
P_{g} & = & \underset{0\le i<j<n}{\prod}\left(x_{j}-x_{i}\right)\underset{0\le u<v\le f\left(n-1\right)}{\prod}\left((x_{f\left(v\right)}-x_{v})^{2}-(x_{f\left(u\right)}-x_{u})^{2}\right)\underset{\begin{array}{c}
v\in f^{-1}\left(\left\{ f\left(n-1\right)\right\} \right)\\
0\le u\le f\left(n-1\right)\\
t\in\left\{ 0,1\right\} 
\end{array}}{\prod}\left({\color{red}b_{u,v,t}}+{\color{blue}a_{f\left(n-1\right)}}\right)\times\\
\\
 &  & \underset{\begin{array}{c}
v\in f^{-1}\left(\left\{ f\left(n-1\right)\right\} \right)\\
f\left(n-1\right)<u<v
\end{array}}{\prod}\left({\color{red}b_{v}}+{\color{red}b_{u}}+2{\color{blue}a_{f\left(n-1\right)}}\right)\underset{\begin{array}{c}
v\in f^{-1}\left(\left\{ f\left(n-1\right)\right\} \right)\\
f\left(n-1\right)<u<v
\end{array}}{\prod}\left({\color{red}b_{v}}-{\color{red}b_{u}}+0{\color{blue}a_{f\left(n-1\right)}}\right).
\end{array}\label{telescoping}
\end{equation}
where ${\color{blue}a_{f\left(n-1\right)}}={\color{blue}(}x_{f^{\left(2\right)}\left(n-1\right)}{\color{blue}-x_{f\left(n-1\right)})}$,
${\color{red}b_{i}}={\color{red}(x_{f\left(i\right)}}-x_{i}{\color{red})}$
for all $i\in f^{-1}\left(\left\{ f\left(n-1\right)\right\} \right)$
and
\[
{\color{red}b_{u,v,t}}={\color{red}(x_{f\left(v\right)}}-x_{v}{\color{red})}+\left(-1\right)^{t}(x_{f\left(u\right)}-x_{u}),\ \forall\:\begin{array}{c}
v\in f^{-1}\left(\left\{ f\left(n-1\right)\right\} \right)\\
0\le u\le f\left(n-1\right)\\
t\in\left\{ 0,1\right\} 
\end{array}.
\]
Invoking the \textbf{multi--binomial} identity on the two bichromatic
factors of $P_{g}$ immediately above yields equalities
\[
\prod_{\begin{array}{c}
v\in f^{-1}\left(\left\{ f\left(n-1\right)\right\} \right)\\
f\left(n-1\right)<u<v
\end{array}}\left({\color{red}b_{v}}+{\color{red}b_{u}}+2{\color{blue}a_{f\left(n-1\right)}}\right)=
\]
\[
\bigg(\prod_{\begin{array}{c}
v\in f^{-1}\left(\left\{ f\left(n-1\right)\right\} \right)\\
f\left(n-1\right)<u<v
\end{array}}\big({\color{red}b_{v}}+{\color{red}b_{u}}\big)+\sum_{\begin{array}{c}
r_{u,v}\in\left\{ 0,1\right\} \\
0=\prod r_{u,v}
\end{array}}\prod_{\begin{array}{c}
v\in f^{-1}\left(\left\{ f\left(n-1\right)\right\} \right)\\
f\left(n-1\right)<u<v
\end{array}}\big({\color{red}b_{v}}+{\color{red}b_{u}}\big)^{r_{u,v}}\,\big(2{\color{blue}a_{f\left(n-1\right)}}\big)^{1-r_{u,v}}\bigg)
\]
and 
\[
\prod_{\begin{array}{c}
v\in f^{-1}\left(\left\{ f\left(n-1\right)\right\} \right)\\
0\le u\le f\left(n-1\right)\\
t\in\left\{ 0,1\right\} 
\end{array}}\left({\color{red}b_{u,v,t}}+{\color{blue}a_{f\left(n-1\right)}}\right)=
\]
\[
\bigg(\prod_{\begin{array}{c}
t\in\left\{ 0,1\right\} \\
v\in f^{-1}\left(\left\{ f\left(n-1\right)\right\} \right)\\
0\le u\le f\left(n-1\right)
\end{array}}{\color{red}b_{u,v,t}}+\sum_{\begin{array}{c}
s_{u,v,t}\in\left\{ 0,1\right\} \\
0=\prod s_{u,v,t}
\end{array}}\prod_{\begin{array}{c}
v\in f^{-1}\left(\left\{ f\left(n-1\right)\right\} \right)\\
0\le u\le f\left(n-1\right)\\
t\in\left\{ 0,1\right\} 
\end{array}}\big({\color{red}b_{u,v,t}}\big)^{s_{u,v,t}}\,\big({\color{blue}a_{f\left(n-1\right)}}\big)^{1-s_{u,v,t}}\bigg).
\]
Substituting equalities immediately above into equation (\ref{telescoping})
yields and expression of $P_{g}$ of the form
\[
P_{g}={\color{red}P_{f}}+R_{f,g}.
\]
The first part of the expansion is equal to ${\color{red}P_{f}}$
and collects the monochromatic red parts of the multi-binomial summation
and is given by
\[
\begin{array}{cc}
{\color{red}P_{f}}= & \underset{0\le i<j<n}{\prod}\left(x_{j}-x_{i}\right)\underset{0\le u<v\le f\left(n-1\right)}{\prod}\left((x_{f\left(v\right)}-x_{v})^{2}-(x_{f\left(u\right)}-x_{u})^{2}\right)\underset{\begin{array}{c}
v\in f^{-1}\left(\left\{ f\left(n-1\right)\right\} \right)\\
f\left(n-1\right)<u<v
\end{array}}{\prod}\left({\color{red}b_{v}}-{\color{red}b_{u}}\right)\times\\
 & \bigg(\underset{\begin{array}{c}
v\in f^{-1}\left(\left\{ f\left(n-1\right)\right\} \right)\\
f\left(n-1\right)<u<v
\end{array}}{\prod}\big({\color{red}b_{v}}+{\color{red}b_{u}}\big)\bigg)\,\bigg(\underset{\begin{array}{c}
t\in\left\{ 0,1\right\} \\
v\in f^{-1}\left(\left\{ f\left(n-1\right)\right\} \right)\\
0\le u\le f\left(n-1\right)
\end{array}}{\prod}{\color{red}b_{u,v,t}}\bigg).
\end{array}
\]
 The second part denoted $R_{f,g}$ simply collects all other remaining
multi-binomial summands and is given by
\[
\begin{array}{ccc}
R_{f,g} & = & \underset{0\le i<j<n}{\prod}\left(x_{j}-x_{i}\right)\underset{0\le u<v\le f\left(n-1\right)}{\prod}\left((x_{f\left(v\right)}-x_{v})^{2}-(x_{f\left(u\right)}-x_{u})^{2}\right)\underset{\begin{array}{c}
v\in f^{-1}\left(\left\{ f\left(n-1\right)\right\} \right)\\
f\left(n-1\right)<u<v
\end{array}}{\prod}\left({\color{red}b_{v}}-{\color{red}b_{u}}\right)\times\\
 &  & \left[\bigg(\underset{\begin{array}{c}
v\in f^{-1}\left(\left\{ f\left(n-1\right)\right\} \right)\\
f\left(n-1\right)<u<v
\end{array}}{\prod}\big({\color{red}b_{v}}+{\color{red}b_{u}}\big)\bigg)\right.\bigg(\underset{\begin{array}{c}
s_{u,v,t}\in\left\{ 0,1\right\} \\
0=\prod s_{u,v,t}
\end{array}}{\sum}\underset{\begin{array}{c}
v\in f^{-1}\left(\left\{ f\left(n-1\right)\right\} \right)\\
0\le u\le f\left(n-1\right)\\
t\in\left\{ 0,1\right\} 
\end{array}}{\prod}\big({\color{red}b_{u,v,t}}\big)^{s_{u,v,t}}\,\big({\color{blue}a_{f\left(n-1\right)}}\big)^{1-s_{u,v,t}}\bigg)+\\
\\
 &  & \bigg(\underset{\begin{array}{c}
t\in\left\{ 0,1\right\} \\
v\in f^{-1}\left(\left\{ f\left(n-1\right)\right\} \right)\\
0\le u\le f\left(n-1\right)
\end{array}}{\prod}{\color{red}b_{u,v,t}}\bigg)\bigg(\underset{\begin{array}{c}
r_{u,v}\in\left\{ 0,1\right\} \\
0=\prod r_{u,v}
\end{array}}{\sum}\underset{\begin{array}{c}
v\in f^{-1}\left(\left\{ f\left(n-1\right)\right\} \right)\\
f\left(n-1\right)<u<v
\end{array}}{\prod}\big({\color{red}b_{v}}+{\color{red}b_{u}}\big)^{r_{u,v}}\,\big(2{\color{blue}a_{f\left(n-1\right)}}\big)^{1-r_{u,v}}\bigg)+\\
\\
 &  & \bigg(\underset{\begin{array}{c}
s_{u,v,t}\in\left\{ 0,1\right\} \\
0=\prod s_{u,v,t}
\end{array}}{\sum}\underset{\begin{array}{c}
v\in f^{-1}\left(\left\{ f\left(n-1\right)\right\} \right)\\
0\le u\le f\left(n-1\right)\\
t\in\left\{ 0,1\right\} 
\end{array}}{\prod}\big({\color{red}b_{u,v,t}}\big)^{s_{u,v,t}}\,\big({\color{blue}a_{f\left(n-1\right)}}\big)^{1-s_{u,v,t}}\bigg)\times\\
 &  & \left.\bigg(\underset{\begin{array}{c}
r_{u,v}\in\left\{ 0,1\right\} \\
0=\prod r_{u,v}
\end{array}}{\sum}\underset{\begin{array}{c}
v\in f^{-1}\left(\left\{ f\left(n-1\right)\right\} \right)\\
f\left(n-1\right)<u<v
\end{array}}{\prod}\big({\color{red}b_{v}}+{\color{red}b_{u}}\big)^{r_{u,v}}\,\big(2{\color{blue}a_{f\left(n-1\right)}}\big)^{1-r_{u,v}}\bigg)\right]
\end{array}
\]
The color scheme introduced here is meant to help track the location
of telescoping variables. We now proceed with the main \textbf{contradiction
argument}. Assume for the sake of establishing a contradiction that
the contrapositive claim is false i.e. for some $f$ subject to conditions
described in our premise, we have 
\[
0\equiv{\color{red}{\color{red}P_{f}}}\text{ mod}\left\{ \left(x_{i}\right)^{\underline{n}}:i\in\mathbb{Z}_{n}\right\} \;\text{ and }\;0\not\equiv P_{g}\text{ mod}\left\{ \left(x_{i}\right)^{\underline{n}}:i\in\mathbb{Z}_{n}\right\} .
\]
\[
P_{g}={\color{red}P_{f}}+R_{f,g}\implies P_{g}\equiv R_{f,g}\not\equiv0\mod\left\{ \left(x_{i}\right)^{\underline{n}}:i\in\mathbb{Z}_{n}\right\} .
\]
Observe that every summand in $R_{f,g}$ is a multiple of a positive
power of the binomial ${\color{blue}(}x_{f^{\left(2\right)}\left(n-1\right)}{\color{blue}-x_{f\left(n-1\right)})}$.
We focus in particular on the summand within in $R_{f,g}$ which is
a multiple of the largest possible power of the binomial ${\color{blue}(}x_{f^{\left(2\right)}\left(n-1\right)}{\color{blue}-x_{f\left(n-1\right)})}$.
Namely the summand associated with binary exponent assignments
\[
s_{u,v,t}=0,\ \forall\;\begin{array}{c}
v\in f^{-1}\left(\left\{ f\left(n-1\right)\right\} \right)\\
0\le u\le f\left(n-1\right)\\
t\in\left\{ 0,1\right\} 
\end{array}\text{ as well as }r_{u,v}=0,\ \forall\;\begin{array}{c}
v\in f^{-1}\left(\left\{ f\left(n-1\right)\right\} \right)\\
0\le u\le f\left(n-1\right)
\end{array}.
\]
The said summand is
\[
c\,\underset{0\le i<j<n}{\prod}\left(x_{j}-x_{i}\right)\prod_{0\le u<v\le f\left(n-1\right)}\left((x_{f\left(v\right)}-x_{v})^{2}-(x_{f\left(u\right)}-x_{u})^{2}\right)\bigg(\prod_{\begin{array}{c}
v\in f^{-1}\left(\left\{ f\left(n-1\right)\right\} \right)\\
f\left(n-1\right)<u<v
\end{array}}\left({\color{red}b_{v}}-{\color{red}b_{u}}\right)\bigg)\left({\color{blue}a_{f\left(n-1\right)}}\right)^{m},
\]
where
\[
m=\left|\left\{ \begin{array}{c}
v\in f^{-1}\left(\left\{ f\left(n-1\right)\right\} \right)\\
f\left(n-1\right)<u<v
\end{array}\right\} \right|+\left|\left\{ \begin{array}{c}
v\in f^{-1}\left(\left\{ f\left(n-1\right)\right\} \right)\\
0\le u\le f\left(n-1\right)\\
t\in\left\{ 0,1\right\} 
\end{array}\right\} \right|\:\text{ and }\:c=2^{\left|\left\{ \begin{array}{c}
v\in f^{-1}\left(\left\{ f\left(n-1\right)\right\} \right)\\
f\left(n-1\right)<u<v
\end{array}\right\} \right|}.
\]
The said summand is thus given by
\[
c\,\underset{0\le i<j<n}{\prod}\left(x_{j}-x_{i}\right)\prod_{0\le u<v\le f\left(n-1\right)}\left((x_{f\left(v\right)}-x_{v})^{2}-(x_{f\left(u\right)}-x_{u})^{2}\right)\bigg(\prod_{\begin{array}{c}
v\in f^{-1}\left(\left\{ f\left(n-1\right)\right\} \right)\\
f\left(n-1\right)<u<v
\end{array}}\left(x_{u}-x_{v}\right)\bigg){\color{blue}(}x_{f^{\left(2\right)}\left(n-1\right)}{\color{blue}-x_{f\left(n-1\right)})}^{m}.
\]
It follows from the premise $0\not\equiv\left(P_{g}\text{ mod}\left\{ \left(x_{i}\right)^{\underline{n}}:i\in\mathbb{Z}_{n}\right\} \right)$
that the canonical representative of the chosen summand is non--vanishing.
Observe that the factor 
\[
\underset{0\le i<j<n}{\prod}\left(x_{j}-x_{i}\right)\prod_{0\le u<v\le f\left(n-1\right)}\left((x_{f\left(v\right)}-x_{v})^{2}-(x_{f\left(u\right)}-x_{u})^{2}\right)\prod_{\begin{array}{c}
v\in f^{-1}\left(\left\{ f\left(n-1\right)\right\} \right)\\
f\left(n-1\right)<u<v
\end{array}}\left({\color{red}b_{v}}-{\color{red}b_{u}}\right)
\]
is common to every summand in $R_{f,g}$. Factoring the common factor
we write
\[
R_{f,g}\left(\mathbf{x}\right)=\left(\underset{0\le i<j<n}{\prod}\left(x_{j}-x_{i}\right)\prod_{0\le u<v\le f\left(n-1\right)}\left((x_{f\left(v\right)}-x_{v})^{2}-(x_{f\left(u\right)}-x_{u})^{2}\right)\prod_{\begin{array}{c}
v\in f^{-1}\left(\left\{ f\left(n-1\right)\right\} \right)\\
f\left(n-1\right)<u<v
\end{array}}\left({\color{red}b_{v}}-{\color{red}b_{u}}\right)\right)\,Q_{f,g}\left(\mathbf{x}\right),
\]
where
\[
\begin{array}{ccc}
Q_{f,g} & = & \bigg(\underset{\begin{array}{c}
v\in f^{-1}\left(\left\{ f\left(n-1\right)\right\} \right)\\
f\left(n-1\right)<u<v
\end{array}}{\prod}\big({\color{red}b_{v}}+{\color{red}b_{u}}\big)\bigg)\bigg(\underset{\begin{array}{c}
s_{u,v,t}\in\left\{ 0,1\right\} \\
0=\prod s_{u,v,t}
\end{array}}{\sum}\underset{\begin{array}{c}
v\in f^{-1}\left(\left\{ f\left(n-1\right)\right\} \right)\\
0\le u\le f\left(n-1\right)\\
t\in\left\{ 0,1\right\} 
\end{array}}{\prod}\big({\color{red}b_{u,v,t}}\big)^{s_{u,v,t}}\,\big({\color{blue}a_{f\left(n-1\right)}}\big)^{1-s_{u,v,t}}\bigg)+\\
\\
 &  & \bigg(\underset{\begin{array}{c}
t\in\left\{ 0,1\right\} \\
v\in f^{-1}\left(\left\{ f\left(n-1\right)\right\} \right)\\
0\le u\le f\left(n-1\right)
\end{array}}{\prod}{\color{red}b_{u,v,t}}\bigg)\bigg(\underset{\begin{array}{c}
r_{u,v}\in\left\{ 0,1\right\} \\
0=\prod r_{u,v}
\end{array}}{\sum}\underset{\begin{array}{c}
v\in f^{-1}\left(\left\{ f\left(n-1\right)\right\} \right)\\
f\left(n-1\right)<u<v
\end{array}}{\prod}\big({\color{red}b_{v}}+{\color{red}b_{u}}\big)^{r_{u,v}}\,\big(2{\color{blue}a_{f\left(n-1\right)}}\big)^{1-r_{u,v}}\bigg)+\\
\\
 &  & \bigg(\underset{\begin{array}{c}
s_{u,v,t}\in\left\{ 0,1\right\} \\
0=\prod s_{u,v,t}
\end{array}}{\sum}\underset{\begin{array}{c}
v\in f^{-1}\left(\left\{ f\left(n-1\right)\right\} \right)\\
0\le u\le f\left(n-1\right)\\
t\in\left\{ 0,1\right\} 
\end{array}}{\prod}\big({\color{red}b_{u,v,t}}\big)^{s_{u,v,t}}\,\big({\color{blue}a_{f\left(n-1\right)}}\big)^{1-s_{u,v,t}}\bigg)\times\\
 &  & \bigg(\underset{\begin{array}{c}
r_{u,v}\in\left\{ 0,1\right\} \\
0=\prod r_{u,v}
\end{array}}{\sum}\underset{\begin{array}{c}
v\in f^{-1}\left(\left\{ f\left(n-1\right)\right\} \right)\\
f\left(n-1\right)<u<v
\end{array}}{\prod}\big({\color{red}b_{v}}+{\color{red}b_{u}}\big)^{r_{u,v}}\,\big(2{\color{blue}a_{f\left(n-1\right)}}\big)^{1-r_{u,v}}\bigg).
\end{array}
\]
Let 
\[
\Phi(g):=\sigma\in\left\{ \theta\in\text{S}_{n}:G_{\theta g\theta^{-1}}\in\text{GrL}(G_{g})\right\} .
\]
Observe that for each $\sigma\in\Phi(g)$ there is a non--vanishing
integer evaluation
\[
v_{\sigma}=\underset{0\le i<j<n}{\prod}\left(\sigma\left(j\right)-\sigma\left(i\right)\right)\prod_{0\le u<v\le f\left(n-1\right)}\left(\big(\sigma f(v)-\sigma(v)\big)^{2}-\big(\sigma f(u)-\sigma(u)\big)^{2}\right)\prod_{\begin{array}{c}
v\in f^{-1}\left(\left\{ f\left(n-1\right)\right\} \right)\\
f\left(n-1\right)<u<v
\end{array}}\left(\sigma\left(u\right)-\sigma\left(v\right)\right).
\]
More generally
\[
P_{g}(h)=R_{f,g}\left(h\right)=v_{h}\,Q_{f,g}\left(h\right)=\begin{cases}
\begin{array}{cc}
\text{sgn}(\left|hgh^{-1}-\text{id}\right|\circ h)\underset{0\le i<j<n}{\prod}(j-i)\big(j^{2}-i^{2}) & \text{ if }h\in\Phi(g)\\
\\
0 & \text{otherwise}
\end{array},\forall\:h\in\mathbb{Z}_{n}^{\mathbb{Z}_{n}}.\end{cases}
\]
Recall $Q_{f,g}$ is the remaining factor of $R_{f,g}$ excluding
the common factor. Specifically, $Q_{f,g}$ is a polynomial resulting
from the sum over chromatic summands resulting from the multibinomial
expansions. Let us view $Q_{f,g}$ as a sum of $|\Sigma|$ summands
and denote by $Q_{f,g}^{\left[s\right]}$ the summand $1\leq s\leq|\Sigma|$
of $Q_{f,g}$. By Proposition \textbackslash ref\{prop:ring-homomorphism\},
we can write the remainder of $R_{f,g}$ as follows:
\[
R_{f,g}=\sum_{1\leq s\leq\left|\Sigma\right|}\bigg(\sum_{\sigma\in\Phi(g)}v_{\sigma}\cdot Q_{f,g}^{\left[s\right]}\left(\sigma\right)\cdot L_{\sigma}\left({\bf x}\right)\bigg),
\]
Let us denote by $L_{\sigma}\big({\bf x}_{Q}^{\left[s\right]}\big)$
the factors of $L_{\sigma}\left({\bf x}\right)$ associated with variables
present in $Q_{f,g}^{\left[s\right]}$. Then the evaluations of $R_{f,g}$
over the sublattice $\Phi(g)$ cannot be distinguished from evaluations
of the polynomial
\[
\sum_{1\leq s\leq\left|\Sigma\right|}\bigg(\sum_{\sigma\in\Phi(g)}v_{\sigma}\cdot Q_{f,g}^{\left[s\right]}\left(\sigma\right)\cdot L_{\sigma}\big({\bf x}_{Q}^{\left[s\right]}\big)\bigg)
\]
over the same sublattice. By the transposition invariance, the premise
$P_{g}\equiv R_{f,g}$ implies that any transposition $\tau\in\text{S}_{n}$
which exchanges $x_{f(n-1)}$ with $x_{v}$ where $v\in f^{-1}\left(\left\{ f\left(n-1\right)\right\} \right)$
fixes the remainder of $R_{f,g}$, i.e., for all $v\in f^{-1}\left(\left\{ f\left(n-1\right)\right\} \right)$
the transposition $\tau=\left(f(n-1),v\right)$ applied to the variables
in the remainder of $R_{f,g}$ yield a permutation of the non-vanishing
points on the sublattice $\Phi(g)$. By construction, as
\[
\sum_{1\leq s\leq\left|\Sigma\right|}\bigg(\sum_{\sigma\in\Phi(g)}v_{\sigma}\cdot Q_{f,g}^{\left[s\right]}\left(\sigma\right)\cdot L_{\sigma}\big({\bf x}_{Q}^{\left[s\right]}\big)\bigg)
\]
is a sum over the same non-vanishing points, it is fixed by the said
transposition on the sublattice $\Phi(g)$. That is to say
\[
\tau=\left(g(n-1),v\right)\in\text{Aut}\left(\sum_{1\leq s\leq\left|\Sigma\right|}\bigg(\sum_{\sigma\in\Phi(g)}v_{\sigma}\cdot Q_{f,g}^{\left[s\right]}\left(\sigma\right)\cdot L_{\sigma}\big({\bf x}_{Q}^{\left[s\right]}\big)\bigg)\right),
\]
for all $\in f^{-1}\left(\left\{ f\left(n-1\right)\right\} \right)$.
As a consequence of the multi-binomial expansion, the sum expressing
$Q_{f,g}$ features as one of its summand a unique monochromatic blue
binomial summand, say $Q_{f,g}^{\left[1\right]}$, given by
\[
Q_{f,g}^{\left[1\right]}=c\left({\color{blue}a_{f\left(n-1\right)}}\right)^{m}=c\,{\color{blue}(}x_{f^{\left(2\right)}\left(n-1\right)}{\color{blue}-x_{f\left(n-1\right)})}^{m}.
\]
Thus,
\[
Q_{f,g}^{\left[1\right]}\equiv c\,\sum_{\sigma\in\Phi(g)}{\color{blue}\big(}\sigma f^{(2)}(n-1){\color{blue}-\sigma f(n-1)\big)}^{m}\times
\]
\[
\prod_{j_{f^{(2)}\left(n-1\right)}\in\mathbb{Z}_{n}\backslash\left\{ \sigma f^{(2)}\left(n-1\right)\right\} }\left(\frac{x_{f^{(2)}\left(n-1\right)}-j_{f^{(2)}\left(n-1\right)}}{\sigma f^{(2)}\left(n-1\right)-j_{f^{(2)}\left(n-1\right)}}\right)\prod_{j_{f\left(n-1\right)}\in\mathbb{Z}_{n}\backslash\left\{ \sigma f\left(n-1\right)\right\} }\left(\frac{{\color{blue}x_{f\left(n-1\right)}}-j_{f\left(n-1\right)}}{\sigma f\left(n-1\right)-j_{f\left(n-1\right)}}\right).
\]
Let us now focus on the action of a transposition on individual summands
of some polynomial resulting from an arbitrary but fixed partition
of its non-vanishing monomial terms. There are exactly three distinct
ways that a candidate transposition of variables can lie in the automorphism
group of a given polynomial. Assume that we reason about a particular
summand denoted as $S$.

Option 1: The candidate transposition of a pair of variables fixes
the chosen summand $S$. This occurs when $S$ is symmetric in the
chosen pair of variables being transposed.

Option 2: The candidate transposition of the chosen pair of variables
does not fix $S$ (i.e., Option 1 does not apply) but induces in turn
a transposition which exchanges the chosen summand $S$ with some
other summand from the partition say, $S'$. This occurs, for instance,
if we consider the sum $S+S'$ where $S=(x_{0})^{2}\,x_{1}$ and $S'=x_{0}\,(x_{1})^{2}$.
In this example, we see that transposition which exchanges variables
$x_{0}$ with $x_{1}$ does not fix $S$, but it induces a transposition
which exchanges the summand $S$ with the summand $S'$.

Option 3: The candidate transposition of a pair of variables neither
fixes $S$ nor does it induce a transposition which exchanges $S$
with some other summand (i.e., neither Option 1 nor Option 2 applies).
Instead, $S$ is such that a symmetry broadening cancellation occurs.
Such a cancellation must involve interaction between the non-vanishing
monomials within the monomial support of $S$ with the non-vanishing
monomials within the support of other summands. Option 3 occurs, for
instance, if we take $S=-x_{1}$ and $S'=x_{0}+2x_{1}$. We see that
in this example neither Option 1 nor Option 2 applies when the candidate
transposition is the transposition which exchanges variables $x_{0}$
with $x_{1}$. However $S+S'=x_{0}+x_{1}$ is symmetric and thus admits
the said transposition in its automorphism group. This fact is due
to the symmetry broadening cancellation of like terms : $-x_{1}+2x_{1}$.

We see from the Monomial Support Lemma (\ref{MnOvLe}) that the monomial
support of the canonical representative immediately above is not fixed
by any transposition $\tau\in\text{S}_{n}$ which exchanges ${\color{blue}x_{f\left(n-1\right)}}$
with $x_{v}$ where $v\in f^{-1}\left(\left\{ f\left(n-1\right)\right\} \right)$.
This first observation accounts for Option 1. Also note that the canonical
representative of the chosen summand does not exchange with the canonical
representative of any other summands when we exchange ${\color{blue}x_{f\left(n-1\right)}}$
with $x_{v}$ where $v\in f^{-1}\left(\left\{ f\left(n-1\right)\right\} \right)$.
Since non--vanishing canonical representative of every other bi-chromatic
summand in $R1_{f,g}$ depends on $3$ or more variables. This second
observation accounts for Option 2 . We now account for Option 3 and
show that there are no symmetry--broadening cancellations which adjoins
$\tau$ to the automorphism group. By the monomial support Lemma,
such a symmetry broadening cancelation can occur only for Lagrange
bases
\[
\prod_{j_{f^{(2)}\left(n-1\right)}\in\mathbb{Z}_{n}\backslash\left\{ \sigma f^{(2)}\left(n-1\right)\right\} }\left(\frac{x_{f^{(2)}\left(n-1\right)}-j_{f^{(2)}\left(n-1\right)}}{\sigma f^{(2)}\left(n-1\right)-j_{f^{(2)}\left(n-1\right)}}\right)\prod_{j_{f\left(n-1\right)}\in\mathbb{Z}_{n}\backslash\left\{ \sigma f\left(n-1\right)\right\} }\left(\frac{{\color{blue}x_{f\left(n-1\right)}}-j_{f\left(n-1\right)}}{\sigma f\left(n-1\right)-j_{f\left(n-1\right)}}\right).
\]
where $\sigma\in\sigma\in\Phi(g)$ is subject to $\sigma\left(n-1\right)=0$
and $G_{f}$ is such that $1=\left|f^{-1}\left(\left\{ f\left(n-1\right)\right\} \right)\right.$.
For in that setting non-vanishing monomials occurring in the expanded
form of said Lagrange bases summands possibly cancel out non-vanishing
monomials occurring in the expanded form of Lagrange bases expressing
canonical representative of bi-chromatic summands in $R1_{f,g}$ of
the form
\[
\left({\color{red}b_{f(n-1),n-1,t}}\right)^{r}\left({\color{blue}a_{f\left(n-1\right)}}\right)^{s}=\left({\color{red}(x_{f\left(n-1\right)}}-x_{n-1}{\color{red})}+\left(-1\right)^{t}(x_{f^{(2)}\left(n-1\right)}-x_{f\left(n-1\right)})\right)^{r}\left(x_{f^{\left(2\right)}\left(n-1\right)}{\color{blue}-x_{f\left(n-1\right)}}\right)^{s}.
\]
Crucially, by the Monomial Support Lemma as previously mentioned such
cancelations are restricted to permutations $\sigma\in\Phi(g)$ where
$\sigma\left(n-1\right)=0$. This restriction breaks the complementary-labeling
symmetry (\ref{Complementary Symmetry}). Indeed by complementary
labeling symmetry, the canonical representative is up to sign invariant
to the involution prescribed by the map: $x_{i}\mapsto x_{n-1-i}$
for all $i\in\mathbb{Z}_{n}$. But the complementary labeling transformation
maps any Lagrange basis associated with $\sigma\in\sigma\in\Phi(g)$
such that $\sigma\left(n-1\right)=0$ to different Lagrange basis
associated $\sigma^{\prime}\in\sigma\in\Phi(g)$ such that $\sigma^{\prime}\left(n-1\right)=n-1$
and thus negates the symmetry broadening cancelations. For we see
that such a symmetry broadening cancelations which adjoins $\tau$
to the automorphism group of the canonical representative of $R_{f,g}$
would break the complementary labeling symmetry. Thereby resulting
in the contradiction
\[
\tau\notin\text{Aut}\left(\text{Canonical Representative of }P_{g}\right).
\]
We conclude that the desired claim $\emptyset\ne\text{GrL}\left(G_{g}\right)\implies\text{GrL}\left(G_{f}\right)\ne\emptyset,$
holds.
\end{proof}

\section{The Graceful Labeling Theorem}

Equipped with the composition lemma, we settle in the affirmative
the KRR conjecture as stated in theorem (\ref{Graceful_Tree_Theorem}).
In fact we prove that for all $f\in\mathbb{Z}_{n}^{\mathbb{Z}_{n}}$,
the maximum number of distinct induced absolute subtractive edge labels
occurring in a relabeling of the graph of the iterate $f^{\left(\text{o}_{f}\right)}$
(where o$_{f}$ is the order of $f$ i.e. the LCM of cycle lengths
occurring in $G_{f}$) is equal $\left(n+1\right)$ minus the number
of connected components occurring in $G_{f^{\left(\text{o}_{f}\right)}}$.
If $G_{f}$ is connected and $f$ has a fixed point, then o$_{f}$
is equal to one and theorem (\ref{Graceful_Tree_Theorem}) follows
as a special case.
\begin{thm}[The Graceful Labeling Theorem]
\label{General_Graceful_Theorem} For all $f\in\mathbb{Z}_{n}^{\mathbb{Z}_{n}}$,
\[
n+1-\left(\text{number of connected components in }G_{f^{\left(\text{o}_{f}\right)}}\right)=\max_{\sigma\in\text{S}_{n}}\left|\left\{ \left|\sigma f^{\left(\text{o}_{f}\right)}\sigma^{\left(-1\right)}\left(i\right)-i\right|:i\in\mathbb{Z}_{n}\right\} \right|,
\]
where o$_{f}$ denotes the order of $f$ i.e. the LCM of directed
cycle lengths occurring in $G_{f}$.
\end{thm}

\begin{proof}
The claim trivially holds when $f\in$ S$_{n}$, for in that setting
$o_{f}$ is the order of the permutation $f$ and $f^{\left(o_{f}\right)}=$
id. Otherwise if $f\notin S_{n}$ it suffices to show that for all
$f$ subject to the fixed point condition $\left|f^{(n-1)}\left(\mathbb{Z}_{n}\right)\right|=1$
the equality 
\[
n=\max_{\sigma\in\text{S}_{n}}\left|\left\{ \left|\sigma f\sigma^{\left(-1\right)}\left(i\right)-i\right|:i\in\mathbb{Z}_{n}\right\} \right|,
\]
holds. This latter claim follows by repeatedly iterating the composition
lemma described in lemma \ref{CL}. For we see that, given any such
function $f$, the iterate $f^{\left(2^{\left\lceil \log_{2}\left(n-1\right)\right\rceil }\right)}$
is identically constant. As pointed in example (\ref{Star_Example}),
graphs of identically constant functions are graceful. Thus completing
the proof.
\end{proof}

\section*{Appendix}

We explain here in more detail how the edges of a functional digraph
$G_{f}$ are determined by the polynomial construction. The discussion
presented here is taken from \cite{CLARK2024295} and is provided
here for the benefit of the reader.
\begin{lem}[Recovery Lemma]
\label{lem:Recovery Lemma} For an an arbitrary function $f\in\mathbb{Z}_{n}^{\mathbb{Z}_{n}}$,
let 
\[
p_{f}\left(x_{0},\cdots,x_{i},\cdots,x_{n-1}\right):=\prod_{0\le i<j<n}\left(\left(x_{f\left(j\right)}-x_{j}\right)^{2}-\left(x_{f\left(i\right)}-x_{i}\right)^{2}\right).
\]
Suppose $p_{f}\left(x_{0},\cdots,x_{n-1}\right)$ is defined from
some function $f\in\mathbb{Z}_{n}^{\mathbb{Z}_{n}}$, and $p_{f}$
is not identically zero. If $f$ has a fixed point, then 
\[
G_{f}\cup G_{f^{\top}}\backslash\left\{ \left(i,f\left(i\right)\right):i\in\mathbb{Z}\text{ and }f\left(i\right)=i\right\} 
\]
can be determined from $p_{f}$. If $f$ has a fixed point and $G_{f}$
is connected, then $f$ and $G_{f}$ can be determined from $p_{f}$
and the fixed point. Let $S$ denote the set of functions $f\in\mathbb{Z}_{n}^{\mathbb{Z}_{n}}$
such that $f$ has a unique fixed point $0$ and $G_{f}\cup G_{f^{\top}}\backslash\left\{ \left(0,0\right)\right\} $
is connected. The function from $S$ to $\mathbb{Q}\left[x_{0},\cdots,x_{n-1}\right]$
that assigns $p_{f}$ to $f$ is injective.
\end{lem}

\begin{proof}
We show that each factor in a factorization of $p_{f}$ is a quadrinomial
(a linear combination of exactly four distinct variables), a trinomial
(a linear combination of exactly three distinct variables), or a binomial
(a linear combination of exactly two distinct variables), and analyze
how each can occur.
\[
\left(\left(x_{f\left(j\right)}-x_{j}\right)^{2}-\left(x_{f\left(i\right)}-x_{i}\right)^{2}\right)=\left(x_{f\left(j\right)}-x_{j}+x_{f\left(i\right)}-x_{i}\right)\left(x_{f\left(j\right)}-x_{j}-x_{f\left(i\right)}+x_{i}\right).
\]
A factor $x_{f\left(j\right)}-x_{j}-x_{f\left(i\right)}+x_{i}$ or
$x_{f\left(j\right)}-x_{j}+x_{f\left(i\right)}-x_{i}$ has the form
$a+b-c-d$ it is a quadrinomial with $a,b,c,d$ distinct if and only
if $|\{a,b,c,d\}|=4$ i.e. $\left|\left\{ x_{f\left(j\right)},x_{j},x_{f\left(i\right)},x_{i}\right\} \right|=4$.
In this case both $x_{f\left(j\right)}-x_{j}-x_{f\left(i\right)}+x_{i}$
and $x_{f\left(j\right)}-x_{j}+x_{f\left(i\right)}-x_{i}$ are quadrinomials.

The expression $a+b-c-d$ collapses to a binomial if $|\{a,b\}\cap\{c,d\}|=1$
(note that $|\{a,b\}\cap\{c,d\}|=2$ is impossible since $p_{f}$
is not identically zero). Notice that $a+b-c-d$ occurs in two forms
in $p_{f}$:
\[
\left\{ a,b\right\} =\left\{ x_{f(j)},x_{f(i)}\right\} ,\,\left\{ c,d\right\} =\left\{ x_{j},x_{i}\right\} \ \text{ or }\ \left\{ a,b\right\} =\left\{ x_{f(j)},x_{i}\right\} ,\,\left\{ c,d\right\} =\left\{ x_{j},x_{f(i)}\right\} 
\]
First consider the case that $f$ has a (unique) fixed point $u$.
Then for each $j\ne0$ we obtain two copies of the binomial $x_{f\left(j\right)}-x_{j}$
from 
\[
\pm\left(\left(x_{f\left(j\right)}-x_{j}\right)^{2}-\left(x_{f\left(u\right)}-x_{u}\right)^{2}\right)=\pm\left(x_{f\left(j\right)}-x_{j}\right)^{2}
\]
with $+$ if $j>u$ and $-$ otherwise. 

Now assume neither $i$ nor $j$ is fixed by $f$. A binomial-trinomial
pair of factors arises from 
\[
\left(x_{f\left(j\right)}-x_{j}+x_{f\left(i\right)}-x_{i}\right)\left(x_{f\left(j\right)}-x_{j}-x_{f\left(i\right)}+x_{i}\right)
\]
when $\left\{ a,b\right\} =\left\{ x_{f\left(j\right)},x_{f\left(i\right)}\right\} ,\,\left\{ c,d\right\} =\left\{ x_{j},x_{i}\right\} $.
Without loss of generality, we choose $j=f\left(i\right)$. This produces
\[
\pm\left(x_{f^{\left(2\right)}\left(i\right)}-x_{i}\right)\left(x_{f^{\left(2\right)}\left(i\right)}+x_{i}-2x_{f\left(i\right)}\right)
\]
Similarly, a binomial-trinomial pair of factors arises when $\left\{ a,b\right\} =\left\{ x_{f(j)},x_{i}\right\} $,
$\left\{ c,d\right\} =\left\{ x_{j},x_{f(i)}\right\} $, which implies
$f\left(i\right)=f\left(j\right)$. Setting $i<j$, this produces
\[
\left(2x_{f\left(j\right)}-x_{j}-x_{i}\right)\left(x_{i}-x_{j}\right).
\]
We have now described all possible ways binomial factors can occur
in $p_{f}$. Furthermore, a trinomial factor of $p_{f}$ can only
occur in a binomial-trinomial pair. Observe that in each binomial-trinomial
pair, the trinomial has the form $\pm\left(2r-s-t\right)$ and the
associated binomial is of the form $\left(s-t\right)$.

We now take a given polynomial $p_{f}$ that is not identically zero,
with no information about $f$ except that except that $f\in\mathbb{Z}_{n}^{\mathbb{Z}_{n}}$.
Define $h_{f}\left(x_{0},\dots,x_{n-1}\right)$ to be the product
of all the binomials that occur in binomial-trinomial pairs. That
is, $s-t$ is a factor of $h_{f}$ if and only if $2r-s-t$ is a factor
of $p_{f}$ for some $r$. Now define
\[
q_{f}\left(x_{0},\dots,x_{n-1}\right)=\frac{p_{f}\left(x_{0},\dots,x_{n-1}\right)}{h_{f}\left(x_{0},\dots,x_{n-1}\right)},
\]
which is a polynomial. Then $q_{f}$ has no binomial factors if and
only if $f$ does not have a fixed point. Otherwise, $q$ has $2\left(n-1\right)$
binomial factors, which occur in pairs: $\left(x_{k}-x_{\ell}\right)^{2}$.
Then 
\[
E\left(G_{f}\cup G_{f^{\top}}\backslash\left\{ \left(0,0\right)\right\} \right)=\left\{ \left\{ k,\ell\right\} :\left(x_{k}-x_{\ell}\right)^{2}\mbox{ is a factor of }q_{f}\right\} .
\]
\end{proof}

\section*{Acknowledgement}

This paper is based upon work supported by the Defense Technical Information
Center (DTIC) under award number FA8075-18-D-0001/0015. The views
and conclusions contained in this work are those of the authors and
should not be interpreted as necessarily representing the official
policies, either expressed or implied, of the DTIC. We would like
to thank Daniel Krashen and Leslie Hogben and Doron Zeilberger for
insightful comments.

\bibliographystyle{amsalpha}
\bibliography{A_Proof_of_the_KRR_Conjecture}

\end{document}